\documentclass[a4paper,11pt,reqno]{amsart}
\pdfoutput=1
\usepackage{epsfig,url,paralist}
\usepackage{amssymb}
\usepackage[normalem]{ulem}
\usepackage{fullpage,dynkin-diagrams}
\usepackage{amsmath}
\usepackage{tikz}
\usepackage{relsize}
\usepackage{fullpage}
\usepackage{hyperref,verbatim}
\usepackage[capitalize]{cleveref}
\usepackage{graphicx}
\usepackage{setspace,multirow}
\usepackage{enumitem,lineno}
\setlist{nolistsep}
\usepackage{lscape}
\usepackage{calrsfs}            
\usepackage{color}
\usepackage{colortbl}
\usepackage{xcolor}
\usepackage{wrapfig}
\usepackage[font=small]{caption}
\DeclareCaptionLabelFormat{tight}{#1~#2}
\usetikzlibrary{calc, intersections}

\numberwithin{equation}{section}
\newtheorem{theorem}{Theorem}[section]

\newtheorem{theorem*}{Main Result}
\newtheorem{lem}[theorem]{Lemma}
\newtheorem{coro}[theorem]{Corollary}
\newtheorem{cor*}[theorem*]{Corollary}
\newtheorem{prop}[theorem]{Proposition}
\theoremstyle{definition}
\newtheorem{defn}[theorem]{Definition}

\newtheorem{nota}[theorem]{Notation}

\newtheorem{exm}[theorem]{Example}
\newtheorem{observation}[theorem]{Observation}
\newtheorem{fact}[theorem]{Fact}
\numberwithin{theorem}{section}
\usepackage{fancyhdr}
\def\<{\langle}
\def\>{\rangle}
\newcommand{\Res}{\mathsf{Res}}

\newcommand{\per}{\text{\,\footnotesize$\overline{\land}$\,}}
\newcommand{\proj}{\mathsf{proj}}

\newcommand{\Z}{\mathbb{Z}}

\newcommand{\cP}{\mathcal{P}}

\newcommand{\cL}{\mathcal{L}}

\begin{document}

\title{On the Moufang condition for polar spaces}

\author[Sira Busch]{Sira Busch}

\subjclass{20E42}

\keywords{Polar Spaces, Moufang Condition, Root Elations, Root Groups}

\begin{abstract}
We provide direct constructions of nontrivial root elations of thick generalized quadrangles embedded in thick polar spaces of higher finite rank 
and show that all of these can be expressed as self-projectivities of length~$4$. 
This reveals a connection between the little projective groups and the special projectivity groups of thick polar spaces of finite rank explicitly. 
We also prove that the constructed root elations extend to the ambient polar space. 
Furthermore, we introduce point elations and line elations and use them to give an equivalent definition of the Moufang property for thick polar spaces, 
without using the notion of apartments. 
This, in turn, allows us to define the Moufang property for thick polar spaces of infinite rank. 
Consequently, we show that all thick polar spaces of finite and infinite rank are Moufang.
\end{abstract}

\maketitle

\begin{center}
\begin{minipage}[c]{0.833\textwidth}
\begin{small}
\textbf{Author:} Sira Busch \\
\textbf{ORCID:} 0009-0009-0939-6543
\end{small}
\end{minipage}
\end{center}


\section{Introduction}

In his lecture notes from 1974, Jacques Tits noticed that all thick irreducible spherical buildings 
of rank at least three satisfy the Moufang condition, and then provided a proof in \cite{Tits77}. 
This includes all finite-dimensional projective spaces and all polar spaces of finite rank.

The Moufang property is a well-known concept in the theory of spherical buildings, incidence geometry, and projective geometry.
It ensures that a given geometry has a rich automorphism group.
Furthermore, it ensures that, as soon as the rank of the geometry is at least three, the subgroup generated by all root elations is a simple group.
Research on this property goes back to the early 1930s, when Ruth Moufang studied different projective planes and their properties.
In 1933, Ruth Moufang discovered a projective plane defined over the octonions (now known as the \emph{Cayley plane}) 
in which Desargues's theorem does not hold, but instead a weaker version of said theorem; the \emph{little Desargues theorem}.
Consequently, projective planes satisfying this weaker condition were called \emph{Moufang planes},
and Jacques Tits named the Moufang property after Ruth Moufang.

Jacques Tits's unified and elegant approach was completely new, and his results laid the foundation for classifying all thick,
irreducible spherical buildings of rank at least three—and hence all finite-dimensional projective spaces of dimension at least three
and all polar spaces of finite rank at least three.

In his proof, Tits made use of another one of his theorems, known as the \emph{extension theorem} (\cite[Theorem~4.1.2]{SphericalTypeFiniteBNPairs}).
The extension theorem is considered rather technical (cf.~\cite[p.~7]{SphericalTypeFiniteBNPairs}), and therefore
there have been attempts to find alternative proofs of Tits’s results that avoid using it.
Moving away from the generality and focusing more on the specific geometries again might illuminate why Tits’s famous result holds.

In this article, we focus on thick polar spaces and first present direct geometric constructions
for root elations of generalized quadrangles embedded in thick polar spaces of higher finite rank.
These constructions imply the Moufang property for all such quadrangles.
We then show that these elations extend to the ambient polar space, which again implies the Moufang property for the entire space.
Thus, we present a new proof of Tits’s result for thick polar spaces.

Furthermore, we give an equivalent way to define the Moufang condition for thick polar spaces
without using the notion of apartments.
This enables us to describe what the Moufang property looks like for polar spaces of infinite rank.
We then proceed to prove the Moufang property for polar spaces of infinite rank, which has not been done before.

A corollary of our constructions for generalized quadrangles embedded in finite higher-rank polar spaces
is that every root elation can be expressed as a self-projectivity of length~$4$.
Groups of projectivities are another well-known concept in projective and incidence geometry.
Our corollary, together with \cite[Proposition~2.3]{Knarr}, motivated the proof of \cite[Theorem~A]{ProjGrSimpLaced},
which gives a more general statement about the connection between the little projective group of a thick irreducible spherical building $\Delta$ and the special projectivity group of $\Delta$.
Even though we will not deal directly with buildings in this article, we note that our results here, together with \cite{Busch:NonExistenceH3H4}
and the known constructions for buildings of type $\mathsf{A_3}=\mathsf{D_3}$, provide alternative, elementary proofs of Tits's results
on thick, irreducible spherical buildings of rank~$3$, without using Tits’s extension theorem.

It is worth mentioning that a simplified version of Tits's extension theorem exists for polar spaces (cf.~\cite[Theorem~8.5.5]{DiagramGeometry}). 
However, this does not give us a direct construction for root elations and it assumes that we already know that a certain isomorphism exists (cf.~\cite[p.~400,~l.~2]{DiagramGeometry}).

\section{Preliminaries}

We start with the most basic and essential definitions.

\begin{defn}
A \textit{point-line geometry} is a pair $\Delta=(\cP,\cL)$, where $\cP$ is a set and $\cL$ is a set of subsets of $\cP$. The elements of $\cP$ are called \textit{points}, and the elements of $\cL$ are called \textit{lines}. If $p \in \cP$ and $L \in \cL$ with $p\in L$, we say that the point $p$ \emph{lies on} the line $L$, and the line $L$ \emph{contains} the point $p$, or \emph{goes through} $p$. If two points $p$ and $q$ lie on a common line, they are called \textit{collinear}, denoted $p \perp q$. 
For any point $p$ and any subset $P \subseteq \cP$, we define \[p^\perp := \{q \in \cP\mid q \perp p\} \text{ and } P^\perp := \bigcap_{p \in P} p^\perp.\]
A \textit{partial linear space} is a point-line geometry in which every line contains at least three points and where there is a unique line through every pair of distinct collinear points $p$ and $q$, which is then denoted by $pq$. 
\end{defn}
		
\begin{exm}\label{example}
Let $V$ be a vector space of dimension at least $3$. 
Let $\cP$ be the set of $1$-dimensional subspaces of $V$, and let $\cL$ be the set of those subsets of $\cP$ that are contained in a given $2$-dimensional subspace of $V$.
Then $(\cP,\cL)$ is called a \emph{projective space (of dimension $\dim V-1$)} (cf.\ \cite{DiagramGeometry}, \cite{PointsAndLines}).  
\end{exm}

\begin{defn}
Let $\Delta = (\cP,\cL)$ be a partial linear space.  
\begin{enumerate}[label=(\roman*)]

\item A subset $S$ of $\cP$ is called a \textit{subspace} of $\Delta$ if every line $L\in\cL$ that contains at least two points of $S$ is contained in $S$. 
A subspace that intersects every line in at least one point is called a \textit{hyperplane}. 
A hyperplane is called \emph{proper} if it does not consist of the whole point set. 
We usually regard subspaces of $\Delta$ as subgeometries of $\Delta$ in the canonical way.

\item A subspace $S$ in which all points are collinear, or equivalently, for which $S \subseteq S^\perp$, is called a \textit{singular subspace}. 
If, moreover, $S$ is not contained in any other singular subspace, it is called a \textit{maximal singular subspace}. 
A singular subspace is called \textit{projective} if, as a subgeometry, it is a projective space. 
				
\item For a subset $P$ of $\cP$, the \textit{subspace generated by $P$} is denoted $\<P\>_\Delta = \<P\>$ and is defined to be the intersection of all subspaces containing $P$. 
A subspace generated by three mutually collinear points not on a common line is called a \textit{plane}. 
(Note that, in general, this is not necessarily a singular subspace. However, in the cases we will deal with, subspaces generated by pairwise collinear points are singular; in particular, planes will be singular subspaces.) 
\end{enumerate}
\end{defn}

\subsection{Polar spaces}

For polar spaces, we will take the viewpoint of Buekenhout--Shult \cite{FoundationsPolar}, and note that all results in this section are well known and were gathered by Hendrik Van Maldeghem in the book \emph{Polar Spaces} \cite{MonographPolarSpaces}. 

\begin{defn}
A \textit{possibly degenerate polar space} $\Delta$ is a point-line geometry in which every line contains at least three points, and for every point $x$, the set $x^\perp$ is a geometric hyperplane. 
If there exists a point in $\Delta$ collinear to all other points, then we call $\Delta$ \emph{degenerate}.
If, for every point $x$, the set $x^\perp$ is a proper geometric hyperplane, we call $\Delta$ a \emph{(non-degenerate) polar space}. 
If we just talk of a \textit{polar space} in the following, we mean a non-degenerate polar space.
\end{defn}	

A polar space is of \emph{finite rank} if its maximal singular subspaces are generated by a finite number of points. 
The minimal such number, if it exists, is called the \emph{rank of the polar space}, and we usually denote it by $n$, where $n$ is always a natural number.
A \emph{submaximal} singular subspace is a hyperplane of a maximal singular subspace. 
In a polar space $\Delta$, all singular subspaces are either empty, points, lines, or projective spaces (which are of finite dimension if $\Delta$ is of finite rank; cf. \cite[Theorem~1.1]{Pasini:09}, \cite[Theorem~7.3.6 \& Lemma~7.3.8]{PointsAndLines}). 
Consequently, a polar space is a partial linear space. 

A polar space of finite rank is called \emph{top-thin} if every submaximal singular subspace is contained in exactly two maximal singular subspaces. 
A polar space of finite rank is either top-thin, or each submaximal singular subspace is contained in at least three maximal singular subspaces (cf. \cite[Theorem~1.7.1]{MonographPolarSpaces}). 
In the latter case, the polar space is called \emph{thick}. 

In this article, we assume all polar spaces of finite rank to be thick.

For polar spaces, if two points are not contained in a common line, we say that they are \textit{opposite} and write $p \equiv q$. 

Among the polar spaces of finite rank, the ones that have rank $2$ play a special role; they are \emph{generalized quadrangles} (cf. \cite[p.~83]{MonographPolarSpaces}).

\begin{defn}\label{genquaddefn} A \emph{generalized quadrangle} is a partial linear space $\Gamma = (\cP, \cL)$ such that the following axioms are satisfied (cf. \cite[Definition~1.3.1]{GeneralizedPolygons}).
\begin{enumerate}[label=(\roman*)]
\item $\Gamma$ contains no ordinary $k$-gon as a subgeometry for $2 \leq k < 4$.
\item Any two elements of $\cP \cup \cL$ are contained in some ordinary quadrangle in $\Gamma$, and these are also called the \emph{apartments of $\Gamma$}.
\item There exists an ordinary pentagon as a subgeometry in $\Gamma$.
\end{enumerate}
\end{defn}

We can also define \emph{apartments} for polar spaces of higher finite rank and show an analog of \ref{genquaddefn}(ii) in this case.

\begin{lem}
Let $\Delta$ be a polar space of rank $n$. Then we can find $2n$ points $p_{-n},p_{-n+1},\ldots,p_{-1},$ $p_1,p_2,\ldots,p_n$ such that $p_i\perp p_j$ if and only if $i+j\neq 0$, for all $i,j\in\{-n,-n+1,\ldots,-1,1,\ldots,n\}$. 
\end{lem}

\begin{proof}
This is \cite[Construction 1.5.3]{MonographPolarSpaces}. 
\end{proof}

\begin{defn}
The set $\{p_{-n},p_{-n+1},\ldots,p_{-1},p_1,p_2,\ldots,p_n\}$ is called a \emph{polar frame}. An \emph{apartment} of a polar space $\Delta$ is the set of all singular subspaces spanned by the points of a polar frame. Hence, the apartments of a generalized quadrangle are regular quadrangles, the apartments of a polar space of rank $3$ are regular octahedra, and the apartments for polar spaces of higher rank are hyperoctahedra. 
\end{defn}

\begin{prop}
Any two singular subspaces are contained in a common apartment.
\end{prop}

\begin{proof}
This is \cite[Proposition 1.6.10]{MonographPolarSpaces}.
\end{proof}


\subsection{Projectivities} 

For this entire subsection, let $\Delta$ be a polar space of finite rank $n \geq 2$.

\begin{lem} 
If $p$ and $b$ are two opposite points in $\Delta$, then $p^{\perp} \cap b^{\perp}$ is a subspace of $\Delta$. 
If we denote by $\cL_{p,b}$ the set of lines completely contained in $p^{\perp} \cap b^{\perp}$, then $\Delta_{p,b} = (p^{\perp} \cap b^{\perp}, \cL_{p,b})$ is a polar space as soon as $\cL_{p,b}$ is non-empty.
\end{lem}

\begin{proof}
This is \cite[Lemma~2.3.1]{MonographPolarSpaces}.
\end{proof}

\begin{prop}\label{proposition2.10} 
If $U$ is a singular subspace of $\Delta$, then $p^{\perp} \cap U$ is either equal to $U$ or a hyperplane of $U$. 
In the latter case, or if $p \not\in U$, we have $\dim(\< p, p^{\perp} \cap U\>) = \dim(p^{\perp} \cap U)+1$.
\end{prop}

\begin{proof}
This is \cite[Proposition~1.4.1]{MonographPolarSpaces}.
\end{proof}

The last proposition justifies the following important definition.

\begin{defn}
Let $U$ and $V$ be two singular subspaces. We denote the set of points in $U$ that are collinear to all points of $V$ as $\proj_{U}(V)$ and call it the \emph{projection of $V$ onto $U$}. We call $U$ and $V$ opposite if both $\proj_{U} (V)$ and $\proj_{V} (U)$ are empty. 
\end{defn}

\begin{prop} Let $U$ and $V$ be singular subspaces of $\Delta$. Then the following hold:
\begin{enumerate}[label=(\roman*)]
\item $\proj_{U}(V)$ is a subspace of $U$.
\item $\dim (V) - \dim (\proj_V (U)) = \dim (U) - \dim (\proj_U (V))$. 
\item Two singular subspaces $U$ and $V$ are opposite if and only if $\dim (U) = \dim (V)$ and no point of $U$ is collinear to all points of $V$.
\item Let $U$ be disjoint from a maximal singular subspace $M$. Then $U' := \< U, \proj_M U\>$ is the unique maximal singular subspace containing $U$ and intersecting $M$ in a singular subspace of dimension $n - 2 - \dim (U)$. Moreover, $U'$ is the union of all singular subspaces containing $U$ as a hyperplane and intersecting $M$ in at least one point.
\item Each maximal singular subspace $U$ has an opposite in $\Delta$.
\end{enumerate}
\end{prop}

\begin{proof}
This is \cite[Note~1.4.4]{MonographPolarSpaces}, \cite[Proposition~1.4.6]{MonographPolarSpaces}, \cite[Corollary~1.4.7]{MonographPolarSpaces}, \cite[Corollary~1.4.8]{MonographPolarSpaces}, and \cite[Theorem~1.4.9]{MonographPolarSpaces}.
\end{proof}

For our construction, we will want to map from residues to residues and introduce a special way of projecting. For that, we first need the following definition.

\begin{defn}
Let $U$ be a singular subspace of rank at most $n-3$ of a polar space $\Delta$ of rank $n$. Let $X_U$ be the set of all singular subspaces of $\Delta$ with dimension $1+\dim(U)$ which contain $U$. Let, for each singular subspace $V$ containing $U$, $V/U$ be the set of elements of $X_U$ contained in $V$. Let $\Omega_U$ be the set of all such $V/U$, for $V$ ranging through the set of all singular subspaces of $\Delta$ containing $U$. We define a new geometry $\Res_{\Delta} (U)$ with $\cP= X_U$ and $\cL$ the set of $1$-dimensional subspaces of $\Omega_U$ and call it the \emph{residue of $U$ in $\Delta$}. 
\end{defn}

\begin{prop} 
The structure $\Res_{\Delta} (U)$ is a polar space of rank $n - 1 - \dim (U)$. It is thick if and only if $\Delta$ is thick and top-thin if and only if $\Delta$ is top-thin. If $U$ and $U'$ are two opposite singular subspaces of dimension at most $n - 3$, then the polar spaces $\Res_{\Delta} (U)$ and $\Res_{\Delta} (U')$ are isomorphic to each other.
\end{prop}

\begin{proof}
This is \cite[Theorem~1.6.4]{MonographPolarSpaces} and \cite[Corollary~1.6.8]{MonographPolarSpaces}.
\end{proof}

\begin{defn}
We call two subspaces, both contained in $\Res_{\Delta}(U)$ and opposite in $\Res_{\Delta}(U)$, \emph{locally opposite (at $U$)} in $\Delta$.
\end{defn}

Now, let $\Delta$ be a polar space of rank $3$.
Let $p$ and $b$ be two opposite points in $\Delta$. Then 
$$\Res_{\Delta} (p) \cong \Res_{\Delta} (b) \cong p^{\perp} \cap b^{\perp} =: \Gamma$$
defines a generalized quadrangle in $\Delta$. The apartments of $\Gamma$ are ordinary quadrangles. 
Let $u$ be a point in $p^{\perp}$. Then we can view $u$ as a point of $\Res_{\Delta}(p)$ by identifying it with the line $pu$. 
We \emph{project $pu$ to $b$} by projecting $b$ onto $pu$ in the usual way, obtaining a point $u'$ collinear to $b$, and then setting $bu'$ as the image of the \emph{projection of $pu$ onto $b$}. 
Like this, we define a \emph{projection from the residue of $p$ to the residue of $b$}, and we also write $\proj_{b}^{p}$.

\begin{defn}
Let $p_0$, $p_1$, \dots, $p_s$ be points inside a polar space $\Delta$, such that $p_i$ is opposite $p_{i-1}$ and $p_{i+1}$ for $i \in \Z / s\Z$. A composition 
$$ \proj_{p_s}^{p_{s-1}} \circ \dots \circ \proj_{p_1}^{p_0} $$
is called a \emph{projectivity}. We call it a \emph{self-projectivity} if $p_s = p_0$, and \emph{even} if $s$ is even. 
The number $s$ is called the \emph{length of the projectivity}.
\end{defn}

\begin{nota}
For a projectivity
$$ \proj_{p_s}^{p_{s-1}} \circ \proj_{p_{s-1}}^{p_{s-2}}  \circ \dots \circ \proj_{p_2}^{p_1} \circ \proj_{p_1}^{p_0}, $$
we also write
$$p_0 \per p_1 \per p_2 \per \dots \per p_{s-1} \per p_{s},$$
which is shorter and enables us to read from left to right and see immediately that we project from $p_0$ to $p_1$ to $\dots$ to $p_s$.
\end{nota}

\section{Definition of the Moufang property for polar spaces}

Let $\Delta$ be a polar space of rank $n \geq 2$. 
Given an apartment $\mathcal{A}$ spanned by a polar frame $$\{p_{-n},p_{-n+1},\ldots,p_{-1},p_1,p_2,\ldots,p_n\},$$ we can obtain a \emph{half-apartment} $\alpha$ in two different ways:
\begin{enumerate}
\item[(R1)] by removing one point $p_{i}$ of the polar frame and all singular subspaces of $\mathcal{A}$ that contain $p_{i}$.
\item[(R2)] by removing two collinear points $p_{i}$ and $p_{j}$ of the polar frame, all maximal singular subspaces $A$ of $\mathcal{A}$ that contain $p_{i}$ and $p_{j}$, and all singular subspaces contained in them that either contain $p_i$ or $p_j$, or
\end{enumerate}
Half-apartments are also called \emph{roots}. We will say that a root is \emph{of the first kind} if we remove one point of the corresponding polar frame, and \emph{of the second kind} if we remove two points. The \emph{interior $\alpha^{+}$ of a root $\alpha$ of the first kind} is obtained when we also remove all submaximal singular subspaces that were contained in one of the removed maximal singular subspaces. The \emph{interior $\alpha^{+}$ of a root $\alpha$ of the second kind} is obtained when we also remove all maximal singular subspaces that contain either $p_{i}$ or $p_{j}$.

\begin{exm} We consider apartments of a polar space of rank $3$ and mark what is contained on the interior of a root in both cases in gray. 

\begin{center}
\begin{tabular}{>{\centering\arraybackslash} m{4cm}  >{\centering\arraybackslash} m{4cm}}
\begin{tikzpicture}
\begin{scope}[xscale=1.35, yscale=1.3, xshift=3cm]
\coordinate (q) at (1, 0, 0);  
\coordinate (b) at (0, 1, 0);  
\coordinate (n) at (-1, 0, 0); 
\coordinate (p) at (0, -1, 0); 
\coordinate (u) at (0, 0, 0.8);  
\coordinate (d) at (0, 0, -0.8); 

\draw[] (q) -- (b);
\draw[] (b) -- (n);
\draw[] (n) -- (p);
\draw[] (p) -- (q);
\draw[] (q) -- (u);
\draw[] (b) -- (u);
\draw[] (n) -- (u);
\draw[] (p) -- (u);
\draw[opacity=0.3] (q) -- (d);
\draw[opacity=0.3] (b) -- (d);
\draw[opacity=0.3] (n) -- (d);
\draw[opacity=0.3] (p) -- (d);

\node at (q) [circle, fill, inner sep=1pt, label={[label distance=-1mm]above right:{\small \(  \)}}] {};
\node at (b) [circle, fill, inner sep=1pt, label={[label distance=-1mm]above:{\small \( \)}}] {};
\node at (n) [circle, fill, inner sep=1pt, label={[label distance=-1mm]left:{\small \(  \)}}] {};
\node at (p) [circle, fill, inner sep=1pt, label={[label distance=-1mm]below:{\small \(  \)}}] {};
\node at (u) [circle, fill, inner sep=1pt, label={[label distance=-2.5mm]below left:{\small \(  \)}}] {};
\node at (d) [circle, fill, inner sep=1pt, label={[label distance=-2.5mm]above right:{\small \(  \)}}] {};

\fill[black, opacity=0.1] (d) -- (q) -- (b) -- cycle;
\fill[black, opacity=0.1] (u) -- (q) -- (b) -- cycle;
\fill[black, opacity=0.1] (n) -- (d) -- (b) -- cycle;
\fill[black, opacity=0.1] (n) -- (u) -- (b) -- cycle;
\draw[line width=0.6mm, black, opacity=0.4] (b) -- (d);
\draw[line width=0.6mm, black, opacity=0.4] (b) -- (q);
\draw[line width=0.6mm, black, opacity=0.4] (b) -- (n);
\draw[line width=0.6mm, black, opacity=0.4] (b) -- (u);
\node at (b) [circle, fill, inner sep=2.5pt, opacity = 0.4] {};
\end{scope}
\end{tikzpicture}
\captionof{figure}{The interior of a root of the first kind.}

&

\begin{tikzpicture}
\begin{scope}[xscale=1.35, yscale=1.3]
\coordinate (q) at (1, 0, 0);  
\coordinate (b) at (0, 1, 0);  
\coordinate (n) at (-1, 0, 0); 
\coordinate (p) at (0, -1, 0); 
\coordinate (u) at (0, 0, 0.8);  
\coordinate (d) at (0, 0, -0.8); 

\draw[] (q) -- (b);
\draw[] (b) -- (n);
\draw[] (n) -- (p);
\draw[] (p) -- (q);
\draw[] (q) -- (u);
\draw[] (b) -- (u);
\draw[] (n) -- (u);
\draw[] (p) -- (u);
\draw[opacity=0.3] (q) -- (d);
\draw[opacity=0.3] (b) -- (d);
\draw[opacity=0.3] (n) -- (d);
\draw[opacity=0.3] (p) -- (d);
\end{scope}

\node at (q) [circle, fill, inner sep=1pt, label={[label distance=-1mm]above right:{\small \(  \)}}] {};
\node at (b) [circle, fill, inner sep=1pt, label={[label distance=-1mm]above:{\small \( \)}}] {};
\node at (n) [circle, fill, inner sep=1pt, label={[label distance=-1mm]left:{\small \(  \)}}] {};
\node at (p) [circle, fill, inner sep=1pt, label={[label distance=-1mm]below:{\small \(  \)}}] {};
\node at (u) [circle, fill, inner sep=1pt, label={[label distance=-2.5mm]below left:{\small \(  \)}}] {};
\node at (d) [circle, fill, inner sep=1pt, label={[label distance=-2.5mm]above right:{\small \(  \)}}] {};

\fill[black, opacity=0.1] (d) -- (q) -- (b) -- cycle;
\fill[black, opacity=0.1] (d) -- (q) -- (p) -- cycle;
\draw[line width=0.6mm, black, opacity=0.4] (d) -- (p) -- (q) -- (b) -- (d);
\draw[line width=0.6mm, black, opacity=0.4] (d) -- (q);
\node at (d) [circle, fill, inner sep=2.5pt, opacity = 0.4] {};
\node at (q) [circle, fill, inner sep=2.5pt, opacity = 0.4] {};
\end{tikzpicture}
\captionof{figure}{The interior of a root of the second kind.}
\end{tabular}
\end{center}
\end{exm}

\begin{defn}[\textbf{Collineations}]\label{collineation}
An \emph{automorphism} or a \emph{collineation} of a point-line geometry is a permutation of the point set inducing a
permutation of the line set; that is, the set of points on any line is mapped onto the set of points of another line, and each line occurs as such an image. In particular, collineations preserve collinearity. The set of all collineations forms a group called the \emph{collineation group} (cf. \cite[\S 7]{MonographPolarSpaces}).
\end{defn}

\begin{defn}\label{MoufangDefn1}
Let $\Delta$ be a polar space of rank $n \geq 2$. 
If, for every root $\alpha$, the group of automorphisms of $\Delta$ that fix $\alpha^{+}$ pointwise and stabilize all singular subspaces that intersect $\alpha^{+}$ nontrivially acts transitively on the set of apartments containing $\alpha$, then we say that $\Delta$ has the \emph{Moufang property}, or that $\Delta$ is \emph{Moufang}.
\end{defn}

\begin{defn}[\textbf{Root elations}]
The collineations of a polar space that fix the interior of a root pointwise are also called \emph{root elations}.
\end{defn}

Thus, the Moufang property ensures the existence of nontrivial root elations. 
In the following, we want to give descriptions of and names to different automorphisms of polar spaces and express the Moufang property solely through these automorphisms, without needing the notions of apartments and roots.

\begin{defn}[\textbf{Point elations}]
Let $\Delta$ be an arbitrary polar space, and let $L$ and $L'$ be two lines through a point $c$ that are not coplanar. 
A \emph{point elation} (with center $(L, c, L')$) is a collineation of $\Delta$
stabilizing all singular subspaces containing $c$ and fixing $L$ and $L'$ pointwise.
\end{defn}

\begin{defn}[\textbf{Line elations}]
Let $\Delta$ be an arbitrary polar space, let $L$ be a line, and let $p$ and $q$ be two distinct points on $L$. 
A \emph{line elation} (with center $(p, L, q)$) is a collineation of $\Delta$ fixing all
points on $L$ and stabilizing all singular subspaces containing $p$ or $q$ (where
$p$ and $q$ are two distinct points on $L$).
\end{defn}

\begin{defn}\label{MoufangDefn2}
Let $\Delta$ be an arbitrary polar space. We define the following two properties.
\begin{enumerate}[label=(\roman*)]
\item[(M1)] For any point $c$ and any two noncoplanar lines $L$ and $L'$ through $c$, the group of point elations with center $(L, c, L')$ acts transitively on the set of points opposite $c$ in $x^{\perp} \cap x{'}^{\perp}$, where $x \in L \setminus \{c\}$ and $x' \in L' \setminus \{c\}$.
\item[(M2)] For any line $L$ and any two distinct points $p$ and $q$ on $L$, the group of line elations with center $(p, L, q)$ acts transitively on the points of the lines through $p$ that are not coplanar with $L$. 
\end{enumerate}
\end{defn}


\begin{prop}\label{MoufangDefn3}
Let $\Delta$ be a polar space of rank $n$. Then $\Delta$ is Moufang if and only if it satisfies the properties (M1) and (M2). 
\end{prop}

\begin{proof}
{``$\boldsymbol{\Rightarrow}$'':}

Let $\Delta$ be a polar space of rank $n$ satisfying (M1) and (M2). 

\textbf{(Point elations)} 
Let $\alpha_1$ be a root of the first kind, let $b$ be a point in $\alpha_1^+$, let $\pi$ and $\pi'$ be two maximal singular subspaces in $\alpha_1^+$ through $b$, and let $a$ and $a'$ be two noncollinear points not contained in $\alpha_1$, both opposite $b$, and such that $\proj_{\pi}(a)=\proj_{\pi}(a')$ and $\proj_{\pi'}(a)=\proj_{\pi'}(a')$. 
We need to find an elation of $\Delta$ pointwise fixing $\alpha_1^{+}$, stabilizing all singular subspaces that intersect $\alpha_1^+$, and mapping $a$ to $a'$. 
Let $L$ and $L'$ be two lines in $\alpha_1^{+}$ through $b$ in $\pi$ and $\pi'$, respectively. 
According to (M1), we can find a point elation $\theta_1$ with center $(L, c, L')$ mapping $a$ to $a'$. 
Using Tits’s rigidity theorem (cf. \cite[Theorem~4.1.1]{SphericalTypeFiniteBNPairs}), we see that $\theta_1$ is the desired root elation.

\textbf{(Line elations)} Let $\alpha_2$ be a root of the second kind containing two collinear points $p$ and $q$. 
Let $m$ and $m'$ be two collinear points both not collinear with $q$ such that $mm'$ contains $p$. 
In order to show that $\Delta$ is Moufang, we need to find an elation of $\Delta$ pointwise fixing $\alpha_2^{+}$, stabilizing all singular subspaces that intersect $\alpha_2^+$, and mapping $m$ to $m'$. 
According to (M2), there exists a line elation $\theta_2$ with center $(p, pq, q)$ mapping $m$ to $m'$. 
Then $\theta_2$ is the desired root elation.

{``$\boldsymbol{\Leftarrow}$'':}

Let $\Delta$ be a polar space of rank $n$ having the Moufang property. 

\textbf{(Line elations)} Let $L$ be any line, and let $p$ and $q$ be two distinct points on $L$. 
Let $m$ and $m'$ be two collinear points both not collinear with $q$ such that $mm'$ contains $p$. 
We need to find a line elation with center $(p, L, q)$ that maps $m$ to $m'$. 
Let $\alpha_2$ be a root of the second kind containing $L$ on its interior. 
Then there exists a root elation $\theta_2$ fixing $\alpha_2^+$ pointwise, stabilizing all singular subspaces intersecting $\alpha_2^+$, and mapping $m$ to $m'$. 
Using Tits’s rigidity theorem (cf. \cite[Theorem~4.1.1]{SphericalTypeFiniteBNPairs}), in the respective residues of $p$ and $q$, we see that $\theta_2$ is the desired line elation.

\textbf{(Point elations)}  Let $c$ be a point, and let $M$ and $N$ be two noncoplanar lines through $c$. 
Let $a$ and $\hat{a}$ be two points, both opposite $c$, and such that 
$$\proj_M(a)=\proj_M(\hat{a})=:m, \quad \proj_N(a)=\proj_N(\hat{a})=:n.$$ 
We need to find a point elation with center $(M, c, N)$ that maps $a$ to $\hat{a}$.
Note that if the rank of $\Delta$ is $2$, the Moufang property ensures the existence of a root elation, which is exactly the desired point elation.
So suppose $\Delta$ has rank $n \geq 3$. 

Let $U$ be an arbitrary maximal singular subspace through $cm$, and let $V$ be a maximal singular subspace through $cn$ locally opposite at $c$.
Put
\[
U_a := a^\perp \cap U \quad \text{and} \quad V_a = a^\perp \cap V,
\]
and similarly for $\hat{a}$:
\[
U_{\hat{a}} := \hat{a}^\perp \cap U \quad \text{and} \quad V_{\hat{a}} = \hat{a}^\perp \cap V.
\]

\textbf{Case 1.} $U_a = U_{\hat{a}}$ and $V_a = V_{\hat{a}}$.

Then there exist a root $\alpha$ of the first kind containing $c$ in its interior,
and a root elation fixing the interior of $\alpha$ and mapping $a$ to $\hat{a}$.

\textbf{Case 2.} $U_a = U_{\hat{a}}$ and $V_a \neq V_{\hat{a}}$.

We denote $V_a \cap V_{\hat{a}}$ by $Y$. 
Then $Y$ is a subspace of dimension $n-2$.
Select a line $L$ through $c$ in $V$ not intersecting $Y$. 
Let $e$ be the projection of $a$ onto $L$ and $\hat{e}$ the projection of $\hat{a}$ onto $L$.  
These are different, as $L$ does not intersect $Y$.  
Consider  
\[
K = U \cap Y^\perp.
\]
Then $K$ is a line through $c$.
Consider the line elation with center $(c, K, K \cap a^\perp)$ mapping $e$ to $\hat{e}$.  
Then $a$ is sent to a point $a'$ that has the same projection onto $L$ as $\hat{a}$.  
Moreover, $Y$ and $U$ are fixed pointwise.
Hence, with self-explanatory notation, we have $U_{\hat{a}} = U_{a'}$, since $U_a$ is fixed pointwise,
and $V_{\hat{a}} = V_{a'}$, since $\proj_L (a') = \hat{e}$ and $V_{a'}$ is a hyperplane of $V$ not containing $c$ or $e$.

Now we can apply Case 1 to $a'$ and $\hat{a}$.
The composition of root elations is then the desired point elation.

\textbf{Case 3.}  $U_a \neq U_{\hat{a}}$ and $V_a \neq V_{\hat{a}}$.

We denote $U_a \cap U_{\hat{a}}$ by $Z$, and $V_a \cap V_{\hat{a}}$ by $Y$. 

Let $L$ be the projection of $Z$ onto $V$, and $M$ the projection of $Y$ onto $U$.  

\begin{itemize}
\item If $L$ intersects $Y$, then we replace $U$ by $\<L,Z\>$ and $V$ by a maximal singular subspace through $cn$ locally opposite $\<L,Z\>$.
\item If $L$ does not intersect $Y$, but $M$ intersects $Z$, then we replace $V$ by $\<M,Y\>$ and $U$ by a maximal singular subspace through $cm$ locally opposite $\<M,Y\>$.
\end{itemize}

Then we are back in Case (2).  

If $L$ and $Y$ do not intersect, and $M$ and $Z$ do not intersect, $L$ and $M$ are locally opposite at $c$.  
Put
$$ d := \proj_M(a), ~ e := \proj_L(a), ~ \hat{e} := \proj_L(\hat{a}).$$

We can consider the line elation with center $(c, M, d)$ mapping $e$ to $\hat{e}$.
Then $a$ is mapped onto a point $a'$ such that $\proj_V(\hat{a}) = \proj_V(a')$,
and then we continue with Case (2) for the points $a'$ and $\hat{a}$.
\end{proof}

\begin{defn}[\textbf{Moufang property for polar spaces of infinite rank}]\label{MoufangDefnInfty}
Let $\Delta$ be a polar space of infinite rank. Then $\Delta$ is \emph{Moufang} if $\Delta$ satisfies the conditions (M1) and (M2) of Definition \ref{MoufangDefn2}.
\end{defn}

\begin{observation}
The line elations of $\Delta$ are exactly the root elations of $\Delta$ fixing the interior of a root of the second kind pointwise, and we also refer to them as such in the following.
\end{observation}


\section{Constructions of elations for generalized quadrangles}\label{GQ}

Let $\tilde{\Delta}$ be a polar space of rank $n \geq 3$. Let $\Gamma$ be an arbitrary generalized quadrangle in $\Delta$ arising as the intersection of $U^\perp$ with $V^\perp$, for two opposite singular subspaces of dimension $n-3$. Then $\Gamma$ is contained in some polar space $\Delta$ of rank $3$ inside $\tilde{\Delta}$, and there exist two opposite points $p$ and $b$ in $\Delta$ such that $\Gamma = p^{\perp} \cap b^{\perp}$. Furthermore, $\Gamma$ is isomorphic to $\Res_{\Delta}(p)$ and $\Res_{\Delta}(b)$. In the following, we prove the Moufang condition for $\Gamma$. For that, we have to show that the group of automorphisms of $\Gamma$ that pointwise fix the interior of an arbitrary half-apartment $\alpha^{+}$ of an arbitrary apartment $\alpha$ of $\Gamma$ and stabilize all elements incident with some element on the interior of the half-apartment $\alpha^{+}$ acts transitively on the set of apartments containing $\alpha^{+}$. We will give direct constructions for the desired automorphisms for both kinds of roots.

\subsection{Second kind of root}

\begin{lem}\label{GQFirst}
Let ${\Delta}$ be a polar space of rank $3$, and let $\Gamma$ be the generalized quadrangle obtained from two opposite points $p,b$ of $\Delta$ by considering $p^\perp\cap b^\perp$. Then all roots of the second kind of $\Gamma$ are Moufang. 
Furthermore, every line elation of $\Gamma$ can be written as a self-projectivity of length $4$.
\end{lem}

\begin{proof}
Let $\alpha$ be an arbitrary apartment of $\Gamma$. 
Let $q$ and $d$ be two collinear points in $\alpha$. 
Let $u$ be the point in $\alpha$ collinear to $q$ but not to $d$, and let $n$ be the point in $\alpha$ collinear to $d$ but not to $q$. 

Our goal is to construct an automorphism that fixes $qd$ pointwise, stabilizes the lines that intersect $q$ or $d$, and moves the line $nu$ to a line $n'u'$ 
such that $n'\in nd\setminus\{n\}$ and $u'\in uq\setminus\{u\}$. 

Let $j$ be a point on $pq$ not equal to $p$ or $q$. 
The lines $ju$ and $pu'$ intersect in a point that we denote by $i$. 
Since $j \neq q$ and $u' \neq u$, we have $i \not\in uq$, and since $uq=b^{\perp} \cap \<p,u,q\>$, the points $b$ and $i$ are not collinear.

Furthermore, since $j \neq p$ and $u,u' \neq q$, we have $i \not\in pq$, and since $pq=d^{\perp} \cap \<p,u,q\>$, the points $d$ and $i$ are not collinear.
In particular, $i \neq j$.

Set $\ell := \proj_{bd} (i)$. Then $\ell \in bd \setminus \{b,d\}$.
Since $\ell$ is contained in the plane $\<b,d,q\>$, $\ell \not\perp p$, $\proj_{pq}(\ell)=q \neq j$, the points $\ell$ and $j$ are opposite.
For a visualization, see Figure~\ref{firstconstruction}.

\begin{center}
\begin{tikzpicture}
\begin{scope}[xscale=1.35, yscale=1.3]
\coordinate (q) at (1, 0, 0);  
\coordinate (b) at (0, 1, 0);  
\coordinate (n) at (-1, 0, 0); 
\coordinate (p) at (0, -1, 0); 
\coordinate (u) at (0, 0, 0.8);  
\coordinate (d) at (0, 0, -0.8); 

\draw[] (q) -- (b);
\draw[] (b) -- (n);
\draw[] (n) -- (p);
\draw[] (p) -- (q);
\draw[] (q) -- (u);
\draw[] (b) -- (u);
\draw[] (n) -- (u);
\draw[] (p) -- (u);
\draw[opacity=0.3] (q) -- (d);
\draw[opacity=0.3] (b) -- (d);
\draw[opacity=0.3] (n) -- (d);
\draw[opacity=0.3] (p) -- (d);

\coordinate (j) at ($ (p) + 0.3*(q) - 0.3*(p) $); 
\coordinate (l) at ($ (d) + 0.5*(b) - 0.5*(d) $); 
\coordinate (n') at ($ (n) + 0.5*(d) - 0.5*(n) $); 
\coordinate (u') at ($ (u) + 0.5*(q) - 0.5*(u) $); 

\draw[opacity=0.3] (n') -- (u');

\draw[name path=path1, opacity=0.3] (j) -- (u);
\draw[name path=path2, opacity=0.3] (p) -- (u');
    
\path [name intersections={of=path1 and path2, by=i}];
    
\coordinate (i) at (i);
\draw[opacity=0.3] (l) -- (i);
\end{scope}

\node at (q) [circle, fill, inner sep=1pt, label={[label distance=-1mm]above right:{\small \( q \)}}] {};
\node at (b) [circle, fill, inner sep=1pt, label={[label distance=-1mm]above:{\small \( b \)}}] {};
\node at (n) [circle, fill, inner sep=1pt, label={[label distance=-1mm]left:{\small \( n \)}}] {};
\node at (p) [circle, fill, inner sep=1pt, label={[label distance=-1mm]below:{\small \( p \)}}] {};
\node at (u) [circle, fill, inner sep=1pt, label={[label distance=-2.5mm]below left:{\small \( u \)}}] {};
\node at (d) [circle, fill, inner sep=1pt, label={[label distance=-2.5mm]above right:{\small \( d \)}}] {};

\node at (j) [circle, fill, inner sep=1pt, label={[label distance=-1mm]below right :{\small \( j \)}}] {};
\node at (l) [circle, fill, inner sep=1pt, label={[label distance=-1mm]above right:{\small \( \ell \)}}] {};
\node at (n') [circle, fill, inner sep=1pt, label={[label distance=-1mm]above:{\small \( n' \)}}] {};
\node at (u') [circle, fill, inner sep=1pt, label={[label distance=-1mm]below:{\small \( u' \)}}] {};
\node at (i) [circle, fill, opacity=0.3, inner sep=1pt, label={[label distance=1.5mm] below:{\small \( i \)}}] {};
\end{tikzpicture}
\captionof{figure}{The points and lines relevant for the construction of a root elation for the second kind of root of a generalized quadrangle inside a polar space of higher rank.}
\label{firstconstruction}
\end{center}

Thus, the points $p$ and $b$, $b$ and $j$, $j$ and $\ell$, and $\ell$ and $p$ are opposite. 
We define $\theta \colon \Res(p) \rightarrow \Res(p)$ as follows:
\[ \theta :=  \proj_{p}^{\ell} \circ \proj_{\ell}^{j} \circ \proj_{j}^{b} \circ \proj_{b}^{p}. \]
The map in $\Gamma$ defined by $x\mapsto \theta(px)\cap b^\perp$ is a collineation of $\Gamma$ that we will also denote by $\theta$ (and there is no danger of confusion). 

Since $qd$ is in $p^{\perp} \cap b^{\perp} \cap j^{\perp} \cap \ell^{\perp}$, every point of $qd$ is fixed by $\theta$.  

Let $\pi$ be an arbitrary plane through $pq$. Let $\pi'$ be the unique plane through $b$ intersecting $\pi$ in a line. Then $\pi'=\proj_{b}^{p}(\pi)$. Since $j\in\pi$, the unique plane through $j$ intersecting $\pi$ is $\pi$ again. This means that $\proj_{j}^{b} \circ \proj_{b}^{p}(\pi)=\pi$. The same token shows $\proj_{p}^{\ell} \circ \proj_{\ell}^{j}(\pi)=\pi$, and so $\theta(\pi)=\pi$. 

Now let $\pi$ be an arbitrary plane through $pd$. Set $\pi'=\proj_{b}^{p}(\pi)$. As in the previous paragraph, since $\ell\in \pi'$, we find $\proj_{\ell}^{j} \circ \proj_{j}^{b}(\pi')=\pi'$. Projecting $\pi'$ onto $\Res(p)$ yields $\pi$ again. Hence $\theta(\pi)=\pi$ again. 

Intersecting with $\Gamma$, we already see that $\theta$ is a root elation fixing all points on $dq$ and all lines through $d$ and $q$. There remains to show that $\theta$ maps $un$ to $u'n'$.

The line $pu$ maps to $bu$ under the first projection and then to $ju$ under the second projection. Since $i$ is the unique point of $ju$ that $\ell$ is collinear to, the line $ju$ maps to the line $i\ell$ under the third projection and then to $ip$ under the fourth projection. Since $i$ is on $pu'$, we have $pu'=pi$, and with that $\theta(u)=u'$. Projecting preserves incidence, and therefore the unique point on $nd$ that $u'$ is collinear to has to be the point $\theta(n)=n'$. Hence the line $nu$ moves to $n'u'$ under $\theta$.
\end{proof}

\begin{observation}\label{observation1} \em
If we change the position of $\ell$ in our construction and define it to be the projection of the intersection point of $pv'$ and $jv$ onto $bn$, then 
$$\proj_{p}^{\ell} \circ \proj_{\ell}^{j} \circ \proj_{j}^{b} \circ \proj_{b}^{p}$$ 
is an automorphism that stabilizes all planes through $pq$ or $pn$ and moves the point $v \neq u$ on $uq$ to the point $v' \neq u$ on $uq$. For a visualization, see Figure~\ref{secondconstruction}.

\begin{center}
\begin{tikzpicture}
\begin{scope}[xscale=1.35, yscale=1.3]
\coordinate (q) at (1, 0, 0);  
\coordinate (b) at (0, 1, 0);  
\coordinate (n) at (-1, 0, 0); 
\coordinate (p) at (0, -1, 0); 
\coordinate (u) at (0, 0, 0.8);  
\coordinate (d) at (0, 0, -0.8); 

\draw[] (q) -- (b);
\draw[] (b) -- (n);
\draw[] (n) -- (p);
\draw[] (p) -- (q);
\draw[] (q) -- (u);
\draw[] (b) -- (u);
\draw[] (n) -- (u);
\draw[] (p) -- (u);
\draw[opacity=0.3] (q) -- (d);
\draw[opacity=0.3] (b) -- (d);
\draw[opacity=0.3] (n) -- (d);
\draw[opacity=0.3] (p) -- (d);

\coordinate (j) at ($ (p) + 0.3*(q) - 0.3*(p) $); 
\coordinate (l) at ($ (n) + 0.5*(b) - 0.5*(n) $); 
\coordinate (v) at ($ (u) + 0.4*(q) - 0.4*(u) $); 
\coordinate (v') at ($ (u) + 0.6*(q) - 0.6*(u) $); 

\draw[name path=path1, opacity=0.3] (j) -- (v);
\draw[name path=path2, opacity=0.3] (p) -- (v');
    
\path [name intersections={of=path1 and path2, by=i}];
    
\coordinate (i) at (i);
\draw[opacity=0.3] (l) -- (i);
\end{scope}

\node at (q) [circle, fill, inner sep=1pt, label={[label distance=-1mm]above right:{\small \( q \)}}] {};
\node at (b) [circle, fill, inner sep=1pt, label={[label distance=-1mm]above:{\small \( b \)}}] {};
\node at (n) [circle, fill, inner sep=1pt, label={[label distance=-1mm]left:{\small \( n \)}}] {};
\node at (p) [circle, fill, inner sep=1pt, label={[label distance=-1mm]below:{\small \( p \)}}] {};
\node at (u) [circle, fill, inner sep=1pt, label={[label distance=-2.5mm]below left:{\small \( u \)}}] {};
\node at (d) [circle, fill, inner sep=1pt, label={[label distance=-2.5mm]above right:{\small \( d \)}}] {};

\node at (j) [circle, fill, inner sep=1pt, label={[label distance=-1mm]below right :{\small \( j \)}}] {};
\node at (l) [circle, fill, inner sep=1pt, label={[label distance=-1mm]above left:{\small \( \ell \)}}] {};
\node at (v) [circle, fill, inner sep=1pt, label={[label distance=-1mm]below:{\small \( v \)}}] {};
\node at (v') [circle, fill, inner sep=1pt, label={[label distance=-1mm]below:{\small \( v' \)}}] {};
\node at (i) [circle, fill, opacity=0.3, inner sep=1pt, label={[label distance=1.5mm] below:{\small \(  \)}}] {};
\end{tikzpicture}
\end{center}
\captionof{figure}{The changed positions of the points as described in Observation \ref{observation1}.}
\label{secondconstruction}
\end{observation}

\begin{observation}\label{observation2} \em
If we change the position of $\ell$ in our construction and define it to be on the line $bq$, but not equal to $q$, then 
$$\proj_{p}^{\ell} \circ \proj_{\ell}^{j} \circ \proj_{j}^{b} \circ \proj_{b}^{p}$$ 
is an automorphism that fixes every plane containing $pq$ and every line that is contained in either $\langle p,u,q \rangle$ or $\langle p,d,q \rangle$. For a visualization, see Figure~\ref{thirdconstruction}.

\begin{center}
\begin{tikzpicture}
\begin{scope}[xscale=1.35, yscale=1.3]
\coordinate (q) at (1, 0, 0);  
\coordinate (b) at (0, 1, 0);  
\coordinate (n) at (-1, 0, 0); 
\coordinate (p) at (0, -1, 0); 
\coordinate (u) at (0, 0, 0.8);  
\coordinate (d) at (0, 0, -0.8); 

\draw[] (q) -- (b);
\draw[] (b) -- (n);
\draw[] (n) -- (p);
\draw[] (p) -- (q);
\draw[] (q) -- (u);
\draw[] (b) -- (u);
\draw[] (n) -- (u);
\draw[] (p) -- (u);
\draw[opacity=0.3] (q) -- (d);
\draw[opacity=0.3] (b) -- (d);
\draw[opacity=0.3] (n) -- (d);
\draw[opacity=0.3] (p) -- (d);

\coordinate (j) at ($ (p) + 0.3*(q) - 0.3*(p) $); 
\coordinate (l) at ($ (q) + 0.3*(b) - 0.3*(q) $); 

\end{scope}

\node at (q) [circle, fill, inner sep=1pt, label={[label distance=-1mm]above right:{\small \( q \)}}] {};
\node at (b) [circle, fill, inner sep=1pt, label={[label distance=-1mm]above:{\small \( b \)}}] {};
\node at (n) [circle, fill, inner sep=1pt, label={[label distance=-1mm]left:{\small \( n \)}}] {};
\node at (p) [circle, fill, inner sep=1pt, label={[label distance=-1mm]below:{\small \( p \)}}] {};
\node at (u) [circle, fill, inner sep=1pt, label={[label distance=-2.5mm]below left:{\small \( u \)}}] {};
\node at (d) [circle, fill, inner sep=1pt, label={[label distance=-2.5mm]above right:{\small \( d \)}}] {};

\node at (j) [circle, fill, inner sep=1pt, label={[label distance=-1mm]below right :{\small \( j \)}}] {};
\node at (l) [circle, fill, inner sep=1pt, label={[label distance=-1mm]above right:{\small \( \ell \)}}] {};
\end{tikzpicture}

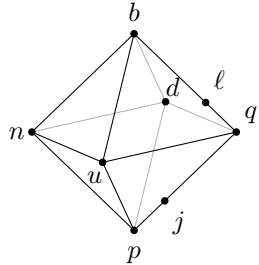
\captionof{figure}{The changed positions of the points as described in Observation \ref{observation2}.}
\label{thirdconstruction}
\end{center}

If $[ \proj_{[\proj_{bn} (j), j]}(\ell), \ell ]$ intersects $pn$, then $\theta$ is the identity. In that case, the lines $pn$, $[\proj_{bn} (j), j]$, $bq$, $pq$, $[\proj_{[\proj_{bn} (j), j]}(\ell), \ell ]$, and $bn$ form a grid.
\end{observation}

The construction in Proposition 8.1 of \cite{ProjGrNonSimpLaced} is a direct generalization of, and based on Observations \ref{observation1} and \ref{observation2}.


\subsection{First kind of root}

\begin{lem}\label{GQSecond}
Let ${\Delta}$ be a polar space of rank $3$, and let $\Gamma$ be the generalized quadrangle obtained from two opposite points $p,b$ of $\Delta$ by considering $p^\perp\cap b^\perp$. Then all roots of the first kind of $\Gamma$ are Moufang.
Furthermore, every root elation of $\Gamma$ fixing the interior of a root of the first kind can be written as a self-projectivity of length $4$.
\end{lem}

\begin{proof}
As in the previous proof, let $\alpha$ be an apartment of $\Gamma$ defined by a polar frame $\{d,q,u,n\}$, where $d$ is collinear to $q$ and $n$, and opposite $u$. Let $\alpha'$ be another apartment of $\Gamma$ spanned by a polar frame $\{d,q,u',n\}$, where $u'$ is a point collinear to $n$ and $q$, and opposite $d$ and $u$. In the following, we will construct an automorphism of $\Gamma$ that fixes $dq$ and $dn$ pointwise, moves $u$ to $u'$, and fixes all lines through $d$.

Pick any point on $bu$ not equal to $b$ or $u$ and denote it by $j'$. Since $b$ and $p$, and $u$ and $u'$ are opposite, and $u\perp p\perp u'\perp b\perp u$, the lines $bu$ and $pu'$ are opposite. Let $j''$ be the projection of $j'$ onto $pu'$. The line $bd$ is opposite $j'j''$, because $j'$ is not collinear to $d$, and if $b$ were collinear to $j''$, then $b$ would be collinear to all points on $u'p$, a contradiction. Let $\ell$ be the projection of $d$ onto $j'j''$. Arguing as above, we see that $bu$ and $pd$ are opposite. Let $j$ be the projection of $j'$ onto $pd$. Since $j' \perp n$ and $j'' \perp n$, we conclude $\ell \perp n$. In the same way, $\ell \perp q$. 

Next, we will show that $\ell$ is opposite both $p$ and $j$. The point $p$ is collinear to $j''$ but not to $j'$, since $pd$ and $bu$ are opposite and $p$ is already collinear to $u$. In particular, $p$ is not collinear to all points of $j'j''$ and not collinear to $\ell$. Suppose that $j$ were collinear to $j''$. Then $j$ would be collinear to all of $pu'$, and $u'$ would be collinear to $p$ and $j$, so to all of $pd$. That contradicts the fact that $u'$ and $d$ are opposite. That means $j'$ is the unique point on $j'j''$ that $j$ is collinear to. In particular, $j$ is not collinear to $\ell$. For a visualization, see Figure~\ref{fourthconstruction}.

\begin{center}
\begin{tikzpicture}
\begin{scope}[xscale=1.65, yscale=1.6]
\coordinate (q) at (1, 0, 0);  
\coordinate (b) at (0, 1, 0);  
\coordinate (n) at (-1, 0, 0); 
\coordinate (p) at (0, -1, 0); 
\coordinate (u) at (0, 0, 0.8);  
\coordinate (d) at (0, 0, -0.8); 

\draw[] (q) -- (b);
\draw[] (b) -- (n);
\draw[] (n) -- (p);
\draw[] (p) -- (q);
\draw[] (q) -- (u);
\draw[] (b) -- (u);
\draw[] (n) -- (u);
\draw[] (p) -- (u);
\draw[opacity=0.3] (q) -- (d);
\draw[opacity=0.3] (b) -- (d);
\draw[opacity=0.3] (n) -- (d);
\draw[opacity=0.3] (p) -- (d);

\coordinate (u') at (0, 0, 2.5);  
\draw[] (p) -- (u');
\draw[] (n) -- (u');
\draw[] (b) -- (u');
\draw[] (q) -- (u');

\coordinate (j) at ($ (p) + 0.8*(d) - 0.8*(p) $); 
\coordinate (j') at ($ (u) + 0.5*(b) - 0.5*(u) $); 
\coordinate (j'') at ($ (u') + 0.8*(p) - 0.8*(u') $); 
\coordinate (l) at ($ (j') + 1.2*(j'') - 1.2*(j') $); 
\draw[opacity=0.3] (j) -- (j') -- (l);

\end{scope}

\node at (q) [circle, fill, inner sep=1pt, label={[label distance=-1mm]above right:{\small \( q \)}}] {};
\node at (b) [circle, fill, inner sep=1pt, label={[label distance=-1mm]above:{\small \( b \)}}] {};
\node at (n) [circle, fill, inner sep=1pt, label={[label distance=-1mm]left:{\small \( n \)}}] {};
\node at (p) [circle, fill, inner sep=1pt, label={[label distance=-1mm]below:{\small \( p \)}}] {};
\node at (u) [circle, fill, inner sep=1pt, label={[label distance=-2.5mm]below left:{\small \( u \)}}] {};
\node at (d) [circle, fill, inner sep=1pt, label={[label distance=-2.5mm]above right:{\small \( d \)}}] {};

\node at (u') [circle, fill, inner sep=1pt, label={[label distance=-1mm]below :{\small \( u' \)}}] {};

\node at (j) [circle, fill, inner sep=1pt, label={[label distance=-1mm] right :{\small \( j \)}}] {};
\node at (j') [circle, fill, inner sep=1pt, label={[label distance=-1mm] above right :{\small \( j' \)}}] {};
\node at (j'') [circle, fill, inner sep=1pt, label={[label distance=-1mm] below left:{\small \( j'' \)}}] {};
\node at (l) [circle, fill, inner sep=1pt, label={[label distance=-1mm]below:{\small \( \ell \)}}] {};
\end{tikzpicture}

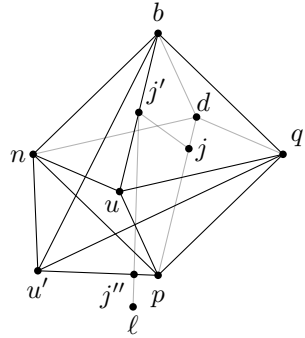
\captionof{figure}{The points and lines relevant for the construction of a root elation for the first kind of root of a generalized quadrangle inside a polar space of higher rank.}
\label{fourthconstruction}
\end{center}

Since $p\equiv b\equiv j\equiv\ell\equiv p$, we can project in each case from one residue into the other. 

We define $\theta \colon \Res(p) \rightarrow \Res(p)$ as follows:
\[ \theta :=  \proj_{p}^{\ell} \circ \proj_{\ell}^{j} \circ \proj_{j}^{b} \circ \proj_{b}^{p} \]
%
From the discussion above, it follows that both $j$ and $\ell$ are collinear to all points of the lines $dq$ and $dn$. This readily implies that $\theta$ fixes every point on $dq\cup dn$. Also, since $j$ is contained in each plane through the line $pd$, we find that each plane through $pd$ is stabilized by $\theta$. Hence $\theta$ induces a root elation in $\Gamma$. We now check that $\theta(u)=u'$. 

We have $bu=\proj_b^p(pu)$. Since $j\perp j'\in bu$, we have $jj'=\proj_j^b(bu)$. Since $j'\perp \ell$, we have $j'\ell=\proj_\ell^j(jj')$. Finally, since $p\perp j''\in j'\ell$, we have $pj''=\proj_p^\ell(j'\ell)$. But $j''\in pu'$, hence $pj''=pu'$. It follows that $\theta(pu)=pu'$, and so $\theta(u)=u'$. This proves the lemma.
\end{proof}

\begin{coro}\label{GQMoufang}
Let ${\Delta}$ be a polar space of rank at least $3$, and let $\Gamma$ be the generalized quadrangle obtained from two opposite singular subspaces $U,V$ of $\Delta$ of dimension $n-3$ by considering $U^\perp\cap V^\perp$. Then $\Gamma$ is a Moufang quadrangle. 
\end{coro}

\begin{proof}
The claim follows from Lemmas \ref{GQFirst} and \ref{GQSecond}.
\end{proof}


\section{Constructions of elations for polar spaces}
 
In the following, we show that the elations we found for generalized quadrangles in Section \ref{GQ} extend to elations of the ambient polar space. We will only show this in detail for the case that the latter has rank $3$. However, it can be done for arbitrary rank $n \geq 3$, and we will comment on that at the end of this section.

\subsection{Second kind of root}\label{LineElationsExtRank3}

For the rest of the section, let $\Delta$ be a polar space of rank $3$. We want to show that, for every line $dq$ in $\Delta$ through two arbitrary collinear points $d$ and $q$, we can construct an automorphism $\phi \colon \Delta \rightarrow \Delta$ that fixes all planes through $dq$ pointwise, all lines through either $d$ or $q$ linewise, and acts transitively on the points of the lines through either $d$ or $q$ that are not contained in a plane through $dq$.

\begin{lem}\label{lemma1}
Let $d, q$ be two collinear points in the polar space $\Delta$. Let $m, m'$ be two collinear points of $\Delta$ not collinear to $q$ but such that $d \in mm'$. Then there exists a unique permutation $\eta$ of the set of points $d^\perp \cup q^\perp$ with the following properties.
\begin{compactenum}[$(i)$]
\item $\eta(m)=m'$,
\item all points of $d^\perp \cap q^\perp$ are fixed,
\item the collineation induced by $\eta$ in each plane $\pi$ containing $d$ or $q$, but not both, is a translation of $\pi$ with axis $\proj_{\pi}(dq)$ and center $d$ or $q$, respectively,
\item $\eta$ preserves collinearity.
\end{compactenum}
\end{lem}

\begin{proof}
Let $\pi$ be a given plane containing $dq$. 
Let $\beta$ be the plane spanned by $m$ and $L := \proj_\pi(m)$. Then $\beta$ contains $d$ and $m'$ (cf. Fig.\ \ref{lemma1setup}). 
Since projective planes inside polar spaces are Moufang, there exists an elation $\eta$ of $\beta$ with axis $L$ and center $d$ such that $\eta(m) = m'$.
Now let $J$ be a line through $q$ not collinear to $d$. 
Then $J$ is opposite every line of $\beta$ containing $d$, except for $L$. 
For every line $K \subseteq \beta$, with $d \in K \neq L$, we define an action 
$$\eta_K \colon J \to J, \quad x \mapsto \proj_J^K(\eta(\proj_K^J(x))).$$

\begin{center}
\begin{tabular}{>{\centering\arraybackslash} m{4.1cm}  >{\centering\arraybackslash} m{6.5cm}}
\resizebox{3.6cm}{!}{
\begin{tikzpicture}[scale=1.1]
\begin{scope}[xscale=1.35, yscale=1.3]
\coordinate (o) at (0, 1, 0);  
\coordinate (n) at (-1, 0, 0); 
\coordinate (q) at (1, 0, 0);  
\coordinate (u) at (0, 0, 0.8);  
\coordinate (d) at (0, 0, -0.8); 
\coordinate (m') at ($ (u) + 0.5*(q) - 0.5*(u) $); 
\coordinate (d') at ($ (u) + 0.5*(n) - 0.5*(u) $); 

\draw[] (q) -- (o);
\draw[] (q) -- (u); 
\draw[thick] (o) -- (u);
\draw[] (o) -- (n);

\coordinate (beta) at ($ (m') + 0.3*(o) - 0.3*(m') $); 
\coordinate (pi) at ($ (d') + 0.3*(o) - 0.3*(d') $); 
\coordinate (L) at ($ (u) + 0.6*(o) - 0.6*(u) $); 
\end{scope}

\fill[lightgray, opacity=0.3] (o) -- (u) -- (n) -- cycle; 
\draw[] (u) -- (n); 
\fill[lightgray, opacity=0.1] (o) -- (u) -- (q) -- cycle; 

\node at (o) [circle, fill, inner sep=1pt, label={[label distance=-1mm]above:{\small \(  \)}}] {};
\node at (q) [circle, fill, inner sep=1pt, label={[label distance=-1mm]below:{\small \( m \)}}] {};
\node at (u) [circle, fill, inner sep=1pt, label={[label distance=-1mm]below:{\small \( d \)}}] {};
\node at (m') [circle, fill, inner sep=1pt, label={[label distance=-1mm]below:{\small \( m' \)}}] {};
\node at (n) [circle, fill, inner sep=1pt, label={[label distance=-1mm]below:{\small \( q \)}}] {};

\node at (beta) [label={[label distance=-1mm, color=gray]center:{\small \( \boldsymbol{\beta} \)}}] {};
\node at (pi) [label={[label distance=-1mm, color=gray]center:{\small \( \boldsymbol{\pi} \)}}] {};
\node at (L) [label={[label distance=-1mm]right:{\small \( L \)}}] {};
\end{tikzpicture}}

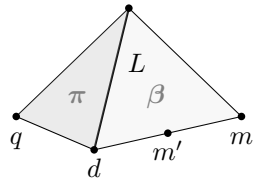
\captionof{figure}{The basic set up for the proof of Lemma \ref{lemma1}.}\label{lemma1setup}

&

\resizebox{5.8cm}{!}{
\begin{tikzpicture}[scale=1.1]
\begin{scope}[xscale=1.35, yscale=1.3]
\coordinate (o) at (0, 1, 0);  
\coordinate (n) at (-1, 0, 0); 
\coordinate (q) at (1, 0, 0);  
\coordinate (u) at (0, 0, 0.8);  
\coordinate (d) at (0, 0, -0.8); 
\coordinate (m') at ($ (u) + 0.5*(q) - 0.5*(u) $); 
\coordinate (d') at ($ (u) + 0.5*(n) - 0.5*(u) $); 
\coordinate (m'') at ($ (u) + 0.9*(q) - 0.9*(u) $); 
\coordinate (K) at ($ (o) + 0.7*(q) - 0.7*(o) $); 

\draw[] (q) -- (o);
\draw[] (q) -- (u); 
\draw[] (o) -- (u);
\draw[] (o) -- (n);
\draw[] (d) -- (n);
\draw[opacity=0.2] (d) -- (o);
\draw[opacity=0.2] (d) -- (m'');
\draw[opacity=0.2] (o) -- (m'');

\coordinate (beta) at ($ (m') + 0.3*(o) - 0.3*(m') $); 
\coordinate (pi) at ($ (d') + 0.3*(o) - 0.3*(d') $); 
\coordinate (L) at ($ (u) + 0.6*(o) - 0.6*(u) $); 
\coordinate (x') at ($ (d) + 0.6*(n) - 0.6*(d) $); 
\end{scope}

\fill[lightgray, opacity=0.3] (o) -- (u) -- (n) -- cycle; 
\draw[] (u) -- (n); 
\fill[lightgray, opacity=0.1] (o) -- (u) -- (q) -- cycle; 

\path[name path=lineA] (o) -- (m'');
\path[name path=lineB] (u) -- (K);
\path[name intersections={of=lineA and lineB, by=I}];

\draw[] (d) -- (I);
\coordinate (K') at ($ (u) + 0.5*(I) - 0.5*(u) $); 
\draw[thick] (u) -- (K);

\coordinate (KL) at ($ (u) + 0.8*(I) - 0.8*(u) $); 

\draw[] (x') -- (K'); 

\node at (o) [circle, fill, inner sep=1pt, label={[label distance=-1mm]above:{\small \(  \)}}] {};
\node at (q) [circle, fill, inner sep=1pt, label={[label distance=-3mm, xshift=1.5mm]below right:{\small \( m \)}}] {};
\node at (u) [circle, fill, inner sep=1pt, label={[label distance=-1mm]below:{\small \( d \)}}] {};
\node at (n) [circle, fill, inner sep=1pt, label={[label distance=-1mm]below:{\small \( q \)}}] {};
\node at (d) [circle, fill, inner sep=1pt, label={[label distance=-2mm]above right:{\small \( x \)}}] {};
\node at (KL) [label={[label distance=-3mm]below right:{\small \( K \)}}] {};
\node at (I) [circle, fill, inner sep=1pt, label={[label distance=-0.5mm]above right:{\small \( \proj_K^J(x) \)}}] {};
\node at (K') [circle, fill, inner sep=1pt, label={[label distance=-1mm, yshift=-1.5mm]below right:{\small \( \eta(\proj_K^J(x)) \)}}] {};
\node at (x') [circle, fill, inner sep=1pt, label={[label distance=0mm, xshift=4.5mm]above left:{\small \( \proj_J^K (\eta(\proj_K^J(x))) \)}}] {};

\node at (beta) [label={[label distance=-1mm, color=gray]center:{\small \( \)}}] {};
\node at (pi) [label={[label distance=-1mm, color=gray]center:{\small \(  \)}}] {};
\node at (L) [label={[label distance=-1mm]right:{\small \(  \)}}] {};
\end{tikzpicture}}
\captionof{figure}{Visualization for the beginning of the proof of Lemma \ref{lemma1}}\label{lemma1defnetaK}
\end{tabular}
\end{center}

We claim
\begin{itemize}
\item[(*)] The map $\eta_K$ is independent of $K$.
\end{itemize}
Indeed, let $K_1$ and $K_2$ be two lines in $\beta$ through $d$, distinct from $L$. 
Let $y_1 \in K_1 \setminus \{d\}$ be arbitrary. 
Set $x := \proj_J(y_1)$. If $z = \proj_L(x)$, then $y_1z = \proj_\beta(x)$. 
It follows that $\proj_{K_2}(x) = y_1z \cap K_2 =: y_2$, and hence $z \in \eta(y_1)\eta(y_2)$, as $\eta(z) = z$ (cf. Fig.\ \ref{lemma1claimstar}). 

\begin{center}\resizebox{4cm}{!}{
\begin{tikzpicture}[scale=1.4]
\begin{scope}[xscale=1.35, yscale=1.3]
\coordinate (o) at (0, 1, 0);  
\coordinate (n) at (-1, 0, 0); 
\coordinate (q) at (1, 0, 0);  
\coordinate (u) at (0, 0, 0.8);  
\coordinate (d) at (0, 0, -0.8); 
\coordinate (m') at ($ (u) + 0.5*(q) - 0.5*(u) $); 
\coordinate (d') at ($ (u) + 0.5*(n) - 0.5*(u) $); 
\coordinate (m'') at ($ (u) + 0.9*(q) - 0.9*(u) $); 
\coordinate (K) at ($ (o) + 0.7*(q) - 0.7*(o) $); 
\coordinate (K2) at ($ (o) + 0.35*(q) - 0.35*(o) $); 

\draw[] (q) -- (o);
\draw[] (q) -- (u); 
\draw[] (o) -- (u);
\draw[] (o) -- (n);
\draw[] (d) -- (n);
\draw[opacity=0.2] (d) -- (o);
\draw[opacity=0.2] (d) -- (q);

\coordinate (beta) at ($ (m') + 0.3*(o) - 0.3*(m') $); 
\coordinate (pi) at ($ (d') + 0.3*(o) - 0.3*(d') $); 
\coordinate (L) at ($ (u) + 0.6*(o) - 0.6*(u) $); 
\coordinate (x') at ($ (d) + 0.6*(n) - 0.6*(d) $); 
\end{scope}

\fill[lightgray, opacity=0.3] (o) -- (u) -- (n) -- cycle; 
\draw[] (u) -- (n); 
\fill[lightgray, opacity=0.1] (o) -- (u) -- (q) -- cycle; 

\path[name path=lineA] (o) -- (m'');
\path[name path=lineB] (u) -- (K);
\path[name intersections={of=lineA and lineB, by=I}];

\draw[] (d) -- (K);
\coordinate (K') at ($ (u) + 0.5*(I) - 0.5*(u) $); 
\draw[thick] (u) -- (K);
\coordinate (KL) at ($ (u) + 0.8*(I) - 0.8*(u) $); 

\draw[] (d) -- (K2);
\draw[thick] (u) -- (K2);
\coordinate (K2L) at ($ (u) + 0.8*(K2) - 0.8*(u) $); 

\coordinate (K'') at ($ (o) + 1.5*(K') - 1.5*(o) $); 
\path[name path=lineA] (o) -- (K'');
\path[name path=lineB] (u) -- (q);
\path[name intersections={of=lineA and lineB, by=I2}];

\draw[] (o) -- (I2);

\path[name path=lineA] (o) -- (K'');
\path[name path=lineB] (u) -- (K2);
\path[name intersections={of=lineA and lineB, by=I3}];

\node at (o) [circle, fill, inner sep=1pt, label={[label distance=-1mm]above:{\small \( z \)}}] {};
\node at (q) [circle, fill, inner sep=1pt, label={[label distance=-2mm]below right:{\small \(  \)}}] {};
\node at (u) [circle, fill, inner sep=1pt, label={[label distance=-1mm]below:{\small \( d \)}}] {};
\node at (n) [circle, fill, inner sep=1pt, label={[label distance=-1mm]below:{\small \( q \)}}] {};
\node at (d) [circle, fill, inner sep=1pt, label={[label distance=-2mm, xshift=1mm]above right:{\small \( x \)}}] {};
\node at (KL) [label={[label distance=-3mm]below right:{\small \( K_1 \)}}] {};
\node at (K) [circle, fill, inner sep=1pt, label={[label distance=-0.5mm]above right:{\small \( y_1 \)}}] {};
\node at (K') [circle, fill, inner sep=1pt, label={[label distance=-1mm, yshift=-1.5mm]below right:{\small \( \eta(y_1) \)}}] {};
\node at (K2) [circle, fill, inner sep=1pt, label={[label distance=-0.5mm]above right:{\small \( y_2 \)}}] {};
\node at (K2L) [label={[label distance=-3.5mm, yshift=3mm, xshift=-0.7mm]above:{\small \( K_2 \)}}] {};
\node at (I3) [circle, fill, inner sep=1pt, label={[label distance=-1mm, yshift=1mm]left:{\small \( \eta(y_2) \)}}] {};

\coordinate (M) at ($ (d) + 0.5*(n) - 0.5*(d) $); 
\node at (M) [label={[label distance=-2mm, yshift=-1.5mm]above left:{\small \( J \)}}] {};

\node at (beta) [label={[label distance=-1mm, color=gray]center:{\small \( \)}}] {};
\node at (pi) [label={[label distance=-1mm, color=gray]center:{\small \(  \)}}] {};
\node at (L) [label={[label distance=-1mm]right:{\small \(  \)}}] {};
\end{tikzpicture}}

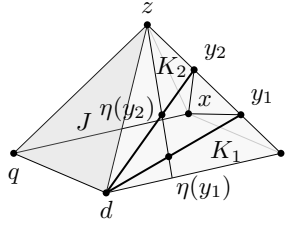
\captionof{figure}{Visualization for the proof of claim (*) in the proof of Lemma \ref{lemma1}}\label{lemma1claimstar}
\end{center}

Since $z = \proj_L(x)$, $z$ is collinear to both $x$ and $q$, and therefore $z$ is collinear to all points of $J$. 
Thus, we have $z = L \cap \proj_\pi(J)$, and as such, $z$ is also equal to $\proj_L(\eta_{K_1}(x)) = \proj_L(\eta_{K_2}(x))$. 

It follows that $\proj_\beta(\eta_{K_1}(x)) = \eta(y_1)z$, which, by the above, coincides with $\eta(y_2)z$. 
Hence, $\eta_{K_1}(x)$ is collinear to $\eta(y_2)$, and with that, we conclude that $\eta_{K_2}(x) = \eta_{K_1}(x)$. 
The claim is proved.

From now on, we will abbreviate the definition of $\eta_K$ above by the sentence “we copy the action of $\eta$ on $K$ to $J$.”

Thanks to (*), we can denote $\eta_K$ simply by $\eta$. 
Now let $\alpha$ be any plane through $J$. We claim that
\begin{itemize}
\item[(**)] There exists a unique elation in $\alpha$ with axis $A := \proj_\alpha(dq)$ and center $q$ extending the action of $\eta$.
\end{itemize}
This follows in the same way as (*) by letting $mm'$ play the role of $J$ and $\alpha$ the role of $\beta$ in the proof of (*).

Let $\beta'$ be any plane intersecting $\beta$ in a line through $d$. 
Then (**) with the roles of $d$ and $q$ interchanged implies that there is a unique elation in $\beta'$ with axis $\proj_{\beta'}(dq)$ agreeing with $\eta$ on the line $\beta \cap \beta'$. 
By (*), the action of that elation on any line of $\beta'$ through $d$ (not collinear to $q$), copied to $J$, agrees with $\eta$. 
If we keep doing this procedure, then—since in $\Res(d)$ the geometry opposite $dq$ is connected (see \cite{MonographPolarSpaces})—the action of $\eta$ on $J$, copied on every line through $d$ opposite $J$, with, additionally, the identity on $d^\perp \cap q^\perp$, defines a permutation $\eta$ preserving collinearity and inducing a nontrivial elation in every plane through $d$ not containing $q$.

Note that $\eta$, acting on $d^\perp$ as defined in the previous paragraph, is independent of the line $J$ through $q$ in $\alpha$, as follows from (*) with $d$ and $q$ interchanged. 
Likewise, we can now extend $\eta$ from $\alpha$ to $q^\perp$, preserving collinearity, and the action on every line through $q$ not collinear to $d$ is copied from the action of $\eta$ on any line through $d$ not collinear to $q$.

Since “copying the action” just means that collinear pairs of points from $d^\perp \times q^\perp$ are mapped to collinear pairs of points, the lemma is proved.
\end{proof}

We now make the connection between $\eta$ as defined in Lemma \ref{lemma1} and the axial elation defined in Lemma \ref{GQFirst}. 

\begin{lem}\label{lemma2}
Let $d,q$ be two collinear points in the polar space $\Delta$. Let $m,m'$ be two collinear points of $\Delta$ not collinear to $q$, but such that $d\in mm'$. Let $\eta$ be the unique permutation of $d^\perp\cup q^\perp$ with the properties mentioned in \emph{Lemma \ref{lemma1}}. Let $p$ and $b$ be two opposite points both collinear to $d,q$, and $m$. Let $\theta$ be the root elation of $p^\perp\cap b^\perp$ with root $(q,qd,d)$ mapping $m$ to $m'$. Then the action of $\theta$ on every line through $q$ or $d$ in $p^\perp\cap b^\perp$ coincides with the action of $\eta$ on that line. 
\end{lem}

\begin{proof}
Let $j$ be an arbitrary point on $pq\setminus\{p,q\}$, and let $\beta$ be the plane spanned by $b,m$, and $d$. 
Define $\ell$ as the intersection of the line $bd$ with the line joining $m$ with $bm'\cap \proj_\beta(j)$ (cf. Fig.\ \ref{forSecondLemmaRank3FirstKind}). 

\begin{center}
\begin{tikzpicture}
\begin{scope}[xscale=1.35, yscale=1.3]
\coordinate (q) at (1, 0, 0);  
\coordinate (b) at (0, 1, 0);  
\coordinate (n) at (-1, 0, 0); 
\coordinate (p) at (0, -1, 0); 
\coordinate (u) at (0, 0, 0.8);  
\coordinate (d) at (0, 0, -0.8); 

\coordinate (beta) at ($ (b) + 0.5*(n) - 0.5*(b) $); 
\fill[lightgray, opacity=0.3] (b) -- (n) -- (d) -- cycle; 

\draw[] (q) -- (b);
\draw[] (b) -- (n);
\draw[] (n) -- (p);
\draw[] (p) -- (q);
\draw[] (q) -- (u);
\draw[] (b) -- (u);
\draw[] (n) -- (u);
\draw[] (p) -- (u);
\draw[opacity=0.6] (q) -- (d);
\draw[opacity=0.6] (b) -- (d);
\draw[opacity=0.6] (n) -- (d);
\draw[opacity=0.6] (p) -- (d);

\coordinate (j) at ($ (p) + 0.3*(q) - 0.3*(p) $); 
\coordinate (l) at ($ (d) + 0.5*(b) - 0.5*(d) $); 
\coordinate (n') at ($ (n) + 0.5*(d) - 0.5*(n) $); 
\coordinate (u') at ($ (u) + 0.5*(q) - 0.5*(u) $); 

\draw[opacity=0.3] (n') -- (u');

\draw[name path=path1, opacity=0.3] (j) -- (u);
\draw[name path=path2, opacity=0.3] (p) -- (u');    
\path [name intersections={of=path1 and path2, by=i}];    
\coordinate (i) at (i);
\draw[opacity=0.3] (l) -- (i);

\draw[opacity=0.3] (j) -- (d) -- (beta) -- (j);
\draw[opacity=0.3] (b) -- (n');

\draw[name path=path1, opacity=0.3] (d) -- (beta);
\draw[name path=path2, opacity=0.3] (b) -- (n');    
\path [name intersections={of=path1 and path2, by=i2}];    
\coordinate (i2) at (i2);
\draw[opacity=0.3] (l) -- (n);
\end{scope}

\node at (q) [circle, fill, inner sep=1pt, label={[label distance=-1mm]above right:{\small \( q \)}}] {};
\node at (b) [circle, fill, inner sep=1pt, label={[label distance=-1mm]above:{\small \( b \)}}] {};
\node at (n) [circle, fill, inner sep=1pt, label={[label distance=-1mm]below left:{\small \( m \)}}] {};
\node at (p) [circle, fill, inner sep=1pt, label={[label distance=-1mm]below:{\small \( p \)}}] {};
\node at (u) [circle, fill, inner sep=1pt, label={[label distance=-2.5mm]below left:{\small \(  \)}}] {};
\node at (d) [circle, fill, inner sep=1pt, label={[label distance=-2.5mm]above right:{\small \( d \)}}] {};

\node at (j) [circle, fill, inner sep=1pt, label={[label distance=-1mm]below right :{\small \( j \)}}] {};
\node at (l) [circle, fill, inner sep=1pt, label={[label distance=-1mm]above right:{\small \( \ell \)}}] {};
\node at (n') [circle, fill, inner sep=1pt, label={[label distance=-1mm]below left:{\small \( m' \)}}] {};
\node at (u') [circle, fill, inner sep=1pt, label={[label distance=-1mm]below:{\small \( \)}}] {};
\node at (i) [circle, fill, inner sep=1pt, label={[label distance=1.5mm] below:{\small \(  \)}}] {};
\node at (beta) [label={[label distance=-1mm, color=lightgray!20!gray, xshift=-8mm, yshift=-3mm]center:{\small \( \boldsymbol{\beta} \)}}] {};
\node at (i2) [circle, fill, inner sep=1pt, label={[label distance=1.5mm] below:{\small \(  \)}}] {};
\end{tikzpicture}

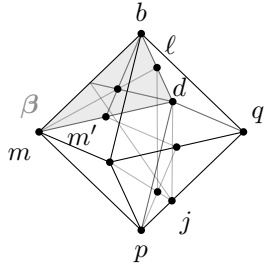
\captionof{figure}{The points and lines relevant to the proof of Lemma \ref{lemma2}.}
\label{forSecondLemmaRank3FirstKind}
\end{center}

With the proof of Lemma \ref{GQFirst}, we find that $p \per \ell \per j \per b \per p$ moves $pm$ to $pm'$. 
Hence, this defines $\theta$ in $p^\perp\cap b^\perp$. 
For $x$ on $mm'$, we set $x':=\ell x\cap \proj_\beta(j)$.
Then the action of $\theta$ on the line $mm'$ is given by $x\mapsto b x'\cap mm'$. 
It is well known that this is the restriction to $mm'$ of the elation of $\beta$ with axis $bd$ and center $d$, mapping $m$ to $m'$. 
Hence, $\theta(x)=\eta(x)$ for all $x\in mm'$. 
Since the actions on the lines of $p^\perp\cap b^\perp$ through $q$ are obtained by copying, for both $\eta$ and $\theta$, these actions coincide. 
By copying these to the lines through $d$, the lemma follows.   
\end{proof}

The following is an immediate consequence of the foregoing. 
\begin{coro}\label{cor1}
Let $d,q$ be two collinear points in the polar space $\Delta$. Let $m,m'$ be two collinear points of $\Delta$ not collinear to $q$, but such that $d\in mm'$. 
Let $\eta$ be the unique permutation of $d^\perp\cup q^\perp$ with the properties mentioned in \emph{Lemma \ref{lemma1}}. 
Let $p$ and $b$ be two opposite points both collinear to $d,q$, and $m$. 
Let $\theta$ be the root elation of $\Res(p)$ with root $(pq,\<p,q,d\>,pd)$ mapping $pm$ to $pm'$. 
Let $\theta_b$ be the collineation of $\Res(b)$ defined by $L\mapsto \proj_b(\theta(\proj_p(L)))$. 
Then $\theta_b$ is the unique root elation of $\Res(b)$ for which the action on the lines through $b$ collinear to $q$ or $d$ coincides with the action of $\eta$. 
\end{coro}

The definition of the mapping $L\mapsto \proj_b(\theta(\proj_p(L)))$ will be abbreviated by “$\theta$ is copied from $p$ to $b$.” 

The next lemma allows us to extend $\eta$ to the whole polar space $\Delta$ in an unambiguous way. 

\begin{lem}\label{lemma3}
Let $q$, $d$, and $\eta$ be as before. Let $x$ be a point collinear to a unique point $y$ of $dq$. 
Let $p$ and $b$ be two opposite points collinear to all of $q$, $d$, and $x$. 
Then there exists a unique root elation $\eta_{p,b}$ in $p^\perp \cap b^\perp$ corresponding to the root $(d,dq,q)$ having the same action as $\eta$ on the points collinear to $q$ or $d$. 
Moreover, $\eta_{p,b}(x)$ does not depend on $p$ and $b$. 
\end{lem}

\begin{proof}
The existence and uniqueness of $\eta_{p,b}$ follow from Lemma \ref{lemma2}. 
Now let $o$ be another point collinear to all of $q,d,x$, and assume first that $o$ is opposite both $p$ and $b$. 
Let $\theta_p$, $\theta_b$, and $\theta_o$ be the root elations of $\Res(p)$, $\Res(b)$, and $\Res(o)$, respectively, for which the action on the lines through $p,b,o$, respectively, collinear to $q$ (or $d$) coincides with the action of $\eta$. 
By Corollary \ref{cor1}, $\theta_p$ copied from $p$ to $b$ is $\theta_b$; $\theta_b$ copied from $b$ to $o$ is $\theta_o$, and $\theta_p$ copied from $p$ to $o$ is again $\theta_o$.
It follows that $\theta_o(ox)$ intersects both $\theta_p(px)$ and $\theta_b(bx)$ nontrivially. 
But $\theta_p(px)=p\eta_{p,b}(x)$ and $\theta_b(bx)=b\eta_{p,b}(x)$. 
Since these two lines are not contained in a plane, the line $\theta_o(ox)$, which intersects both, contains the intersection point $\eta_{p,b}(x)$. 
Hence $\eta_{p,o}(x)=\eta_{p,b}(x)=\eta_{o,b}(x)$. 

Now assume that $o$ is collinear to $b$. 
Note that $o$ is then opposite $p$. 
We denote the plane spanned by $q,d,b$ by $\beta$. 
Since $\Delta$ has rank $3$, $o$ has to be contained in $\beta$, and since $\proj_\beta(x)=by$, we have $o\in by$. 

Since $\eta_{p,b}(x)$ is also collinear to $y$, we find that $o\perp \eta_{p,b}(x)$. 
Hence $\eta_{p,b}(x)=\eta_{p,o}(x)$. 

Now let $p',b'$ be two other opposite points collinear to $p,b,x$. 
At most one of $p,b$ is collinear to $b'$, say $p$ is opposite $b'$. 
By the foregoing, $\eta_{p,b}(x)=\eta_{p,b'}(x)=\eta_{p',b'}(x)$, and the lemma is proved.
\end{proof}

\begin{prop}\label{DefinitionOfRootElationFirstKindOfRootRank3} 
Let $d,q$ be two collinear points in the polar space $\Delta$. Let $m,m'$ be two collinear points of $\Delta$ not collinear to $q$, but such that $d\in mm'$. Using the notation of the previous lemmata, the automorphism $\phi \colon \Delta \rightarrow \Delta$ defined by
\begin{equation*} 
\phi(x) :=\begin{cases}
  x,  & \text{if } x \text{ is in a plane with } dq, \\
  \eta(x), & \text{if } x \perp q \text{ or } x \perp d \text{, not in a plane with } dq, \\
  \eta_{p,b}(x), & \text{if } x \equiv d, x \equiv q,\, p\equiv b,  x,d,q\in p^\perp\cap b^\perp,
\end{cases} 
\end{equation*}
is a well-defined collineation on the whole polar space $\Delta$ that fixes all planes through $dq$ pointwise, stabilizes all lines through either $d$ or $q$, and maps $m$ to $m'$.
\end{prop}

\begin{proof}
We have already proved that $\phi$ is well-defined. Since interchanging the roles of $m$ and $m'$ clearly results in a two-sided inverse map, we deduce that $\phi$ is bijective. It suffices to prove that $\phi$ maps lines into lines. 

This is certainly true for lines intersecting the line $dq$ and for lines collinear to either $q$ or $d$. There are two possibilities left.
\begin{compactenum}[$(i)$]
\item Let $L$ be a line collinear to a unique point $e\in dq\setminus\{d,q\}$. 
The projection $f$ of $q$ onto $L$ is collinear to $dq$. 
Then $f$ is fixed. 
For each point $x\in L\setminus\{f\}$, Lemma \ref{lemma3} ensures that $\phi(x)$ is contained in the line $\theta_f(L)$ (with the notation of the proof of Lemma \ref{lemma3}). 
\item Let $L$ be opposite $dq$. 
We can then pick two opposite points $p,b$ collinear to both $L$ and $qd$. 
Then, by Lemma \ref{lemma3} again, $\phi(L)=\eta_{p,b}(L)$ is a line.  
\end{compactenum} 
The proposition is proved.
\end{proof}
The arbitrariness of $m$ and $m'$ now yields the following consequence.
\begin{coro}\label{PolarSpaceFirst}
Let $\Delta$ be a polar space of rank $3$. Then all roots of the second kind are Moufang.
\end{coro}


\subsection{First kind of root}\label{PointElationsExtRank3}

\begin{lem}\label{PolarSpaceSecond}
Let $\Delta$ be a polar space of rank $3$. Then all roots of the first kind are Moufang.
\end{lem}

\begin{proof}
Let $\alpha$ and $\beta$ be two planes intersecting in a point $o$. Let $L$ and $M$ be two opposite lines in $\alpha$ and $\beta$, respectively, and let $p$ and $p'$ be two points opposite $o$ in $L^{\perp} \cap M^{\perp}$ (cf. Fig.\ \ref{rankthreesecondkindproofsetup}). We aim to construct a collineation $\phi$ of $\Delta$ stabilizing all lines through $o$, fixing $\alpha\cup\beta$ pointwise, and mapping $p$ to $p'$.

Let $x$ be a point on $L$. Let $\eta_x$ be the root elation of $\Res(x)$ fixing all lines of $\alpha$ through $x$, fixing all lines through $x$ that intersect $\beta$, stabilizing all planes through $ox$, and mapping $xp$ to $xp'$. 
We can copy that action onto every point $y\in\beta$ opposite $x$; meaning, we define $\eta_y$ as the unique collineation of $\Res(y)$ mapping any line $K$ through $y$ to $\proj_y(\eta_x(\proj_x K))$, for any $y \in \beta \setminus \proj_{\beta}(x)$ (cf. Fig.\ \ref{definingetay}). 
Note that — since $\eta_x$ fixes all lines through $x$ that intersect $\beta$ — $\eta_y$ fixes $\proj_{\beta}(x)$ pointwise and thus stabilizes all lines through $y$ in $\beta$.

\begin{center}
\begin{tabular}{>{\centering\arraybackslash} m{4.3cm}  >{\centering\arraybackslash} m{5.3cm}}
\resizebox{3.5cm}{!}{
\begin{tikzpicture}[scale=1.2]
\begin{scope}[xscale=1.35, yscale=1.3]
\coordinate (o) at (0, 1, 0);  
\coordinate (n) at (-1, 0, 0); 
\coordinate (q) at (1, 0, 0);  
\coordinate (u) at (0, 0, 0.8);  
\coordinate (d) at (0, 0, -0.8); 
\coordinate (p) at (0, -1, 0); 
\coordinate (p') at (0, -1.5, 0); 

\draw[] (q) -- (o);
\draw[] (o) -- (n);
\draw[] (n) -- (p);
\draw[] (p) -- (q);
\draw[opacity=0.3] (q) -- (u); 
\draw[] (o) -- (u);
\draw[] (n) -- (u);
\draw[] (p) -- (u);
\draw[opacity=0.3] (q) -- (d);
\draw[] (o) -- (d);
\draw[opacity=0.3] (n) -- (d);
\draw[] (p) -- (d);
\draw[] (p') -- (d);
\draw[] (p') -- (n);
\draw[] (p') -- (u);
\draw[] (p') -- (q);

\coordinate (alpha) at ($ (n) + 0.5*(o) - 0.5*(n) $); 
\coordinate (beta) at ($ (q) + 0.5*(o) - 0.5*(q) $); 
\coordinate (L) at ($ (n) + 0.5*(u) - 0.5*(n) $); 
\coordinate (M) at ($ (q) + 0.5*(d) - 0.5*(q) $); 
\end{scope}

\fill[black, opacity=0.1] (o) -- (n) -- (u) -- cycle; 
\fill[black, opacity=0.1] (o) -- (d) -- (q) -- cycle; 
\draw[thick] (n) -- (u); 
\draw[thick] (d) -- (q); 

\node at (o) [circle, fill, inner sep=1pt, label={[label distance=-1mm]above:{\small \( o \)}}] {};
\node at (q) [circle, fill, inner sep=1pt, label={[label distance=-1mm]above right:{\small \(  \)}}] {};
\node at (n) [circle, fill, inner sep=1pt, label={[label distance=-1mm]left:{\small \(  \)}}] {};
\node at (u) [circle, fill, inner sep=1pt, label={[label distance=-2.5mm]below left:{\small \( x \)}}] {};
\node at (d) [circle, fill, inner sep=1pt, label={[label distance=-1.5mm]above right:{\small \( y \)}}] {};
\node at (p) [circle, fill, inner sep=1pt, label={[label distance=-1mm]below:{\small \( p \)}}] {};
\node at (p') [circle, fill, inner sep=1pt, label={[label distance=-1mm]below:{\small \( p' \)}}] {};

\node at (alpha) [label={[label distance=-1mm, color=gray]above left:{\small \( \boldsymbol{\alpha} \)}}] {};
\node at (beta) [label={[label distance=-1mm, color=gray]above right:{\small \( \boldsymbol{\beta} \)}}] {};
\node at (L) [label={[label distance=-1mm, yshift=-1mm, xshift=1mm]above right:{\small \( L \)}}] {};
\node at (M) [label={[label distance=-1mm]above right:{\small \( M \)}}] {};
\end{tikzpicture}}
&
\resizebox{4.4cm}{!}{
\begin{tikzpicture}[scale=1.2]
\begin{scope}[xscale=1.35, yscale=1.3]
\coordinate (o) at (0, 1, 0);  
\coordinate (n) at (-1, 0, 0); 
\coordinate (q) at (1, 0, 0);  
\coordinate (u) at (0, 0, 0.8);  
\coordinate (d) at (0, 0, -0.8); 
\coordinate (p) at (0, -1, 0); 
\coordinate (p') at (0, -1.5, 0); 

\draw[] (q) -- (o);
\draw[] (o) -- (n);
\draw[] (o) -- (u);
\draw[] (n) -- (u);
\draw[] (o) -- (d);

\coordinate (k) at (0, -0.6, -0.1); 
\draw[thick] (d) -- (k);
\coordinate (K) at ($ (k) + 0.5*(d) - 0.5*(k) $); 
\coordinate (XK) at ($ (k) + 0.5*(d) - 0.5*(k) $); 
\draw[] (XK) -- (u);
\coordinate (k') at (0.6, -0.4, 0.1);
\draw[] (k') -- (u);
\coordinate (YK') at ($ (k') + 0.2*(u) - 0.2*(k') $); 
\draw[thick] (YK') -- (d);
\end{scope}

\fill[black, opacity=0.1] (o) -- (n) -- (u) -- cycle; 
\fill[black, opacity=0.1] (o) -- (d) -- (q) -- cycle; 
\draw[thick] (n) -- (u); 
\draw[thick] (d) -- (q); 

\node at (o) [circle, fill, inner sep=1pt, label={[label distance=-1mm]above:{\small \( o \)}}] {};
\node at (q) [circle, fill, inner sep=1pt, label={[label distance=-1mm]above right:{\small \(  \)}}] {};
\node at (n) [circle, fill, inner sep=1pt, label={[label distance=-1mm]left:{\small \(  \)}}] {};
\node at (u) [circle, fill, inner sep=1pt, label={[label distance=-2.5mm]below left:{\small \( x \)}}] {};
\node at (d) [circle, fill, inner sep=1pt, label={[label distance=-1.5mm]above right:{\small \( y \)}}] {};

\node at (k) [circle, fill, inner sep=1pt, label={[label distance=-1mm]below:{\small \(  \)}}] {};
\node at (K) [label={[label distance=-1mm, yshift=3mm, xshift=1mm]left:{\small \( K \)}}] {};
\node at (XK) [circle, fill, inner sep=1pt, label={[label distance=-1mm]below right:{\small \(  \)}}] {};
\node at (k') [circle, fill, inner sep=1pt, label={[label distance=-1mm]right:{\small \( \eta_x(\proj_xK) \)}}] {};
\node at (YK') [circle, fill, inner sep=1pt, label={[label distance=-1mm]above right:{\small \(  \)}}] {};
\end{tikzpicture}} \\
\captionof{figure}{The basic set up for the proof of Lemma \ref{PolarSpaceSecond}.}\label{rankthreesecondkindproofsetup} &  \captionof{figure}{Visualization for the beginning of the proof of Lemma \ref{PolarSpaceSecond}.}\label{definingetay}
\end{tabular}
\end{center}

The collineation $\eta_x$ defines an action on all lines through $x$, and the collineations $\eta_y$ ($y \in \beta \setminus \proj_{\beta} (x)$) “extend” this action to all lines containing some point of $\beta \setminus \proj_{\beta} (x)$. 
We want to show that the action is independent of the choice of $x \in \alpha \setminus \{o\}$.

For that, let $z \in \alpha$ be a point such that $z\neq x$ and such that the line $xz$ does not contain $o$. 
For any $y\in\beta$ opposite $z$, let $\eta_{z,y}$ be the action of $\eta_y$ copied onto $z$. 
Then $\eta_{z,y}$ is always a root elation in $\Res(z)$ with center $oz$ and axes $\alpha$ and $\<z,\proj_\beta(z)\>$. 
Furthermore, each plane through the line $xz$ has the same image under $\eta_x$ and under $\eta_{z,y}$. 
Since a root elation in $\Res(z)$ with center $oz$ is determined by the image of any plane not through $oz$, we conclude that $\eta_z:=\eta_{z,y}$ is independent of $y$. 
Interchanging the roles of $x$ and $z$, we conclude that there exists a unique root elation $\eta_a$ of $\Res(a)$ for each point $a$ of $\alpha$, such that the set of all such root elations, together with all $\eta_b$ for $b\in\beta$, is closed under copying from points to opposite points.

Now let $w$ be an arbitrary point of $\Delta$ opposite $o$. 
Then $w\notin\alpha\cup\beta$, and we set $A:=\proj_\alpha(w)$ and $B:=\proj_\beta(w)$ (cf. Fig.\ \ref{rankthreesecondkindforw}). 
Select $a\in A$ and let $\pi$ be the plane spanned by $w$ and $A$. 
Set $\pi':=\eta_a(\pi)$. 
Select $b\in B$ and let $\sigma$ be the plane spanned by $B$ and $w$. 
By copying, $\sigma':=\eta_b(\sigma)$ intersects $\pi'$ in a unique point, which we define to be $\phi(w)$. 
Note that, by copying, $\eta_a(aw)=a\phi(w)$ for all $a\in A$, and $\eta_b(bw)=b\phi(w)$ for all $b\in B$.

\begin{center}
\begin{tikzpicture}[scale=1.2]
\begin{scope}[xscale=1.35, yscale=1.3]
\coordinate (o) at (0, 1, 0);  
\coordinate (n) at (-1, 0, 0); 
\coordinate (q) at (1, 0, 0);  
\coordinate (u) at (0, 0, 0.8);  
\coordinate (d) at (0, 0, -0.8); 
\coordinate (p) at (0, -1, 0); 
\coordinate (p') at (0, -1.5, 0); 

\draw[] (q) -- (o);
\draw[] (o) -- (n);
\draw[] (o) -- (u);
\draw[] (n) -- (u);
\draw[] (o) -- (d);

\coordinate (wa) at ($ (n) + 0.5*(o) - 0.5*(n) $); 
\draw[thick] (u) -- (wa);

\coordinate (wb) at ($ (q) + 0.5*(o) - 0.5*(q) $); 
\draw[thick] (d) -- (wb);

\coordinate (w) at (0.5, -0.5, -0.5);  
\draw[] (wa) -- (w) -- (wb);
\draw[] (u) -- (w) -- (d);

\fill[black, opacity=0.1] (w) -- (u) -- (wa) -- cycle; 
\fill[black, opacity=0.1] (w) -- (d) -- (wb) -- cycle; 

\coordinate (pi) at (0, -0.5, -0.5);  
\coordinate (sigma) at (0.7, -0.2, -0.5);  

\coordinate (A) at ($ (u) + 0.5*(wa) - 0.5*(u) $); 
\coordinate (B) at ($ (d) + 0.5*(wb) - 0.5*(d) $); 
\end{scope}

\fill[black, opacity=0.1] (o) -- (n) -- (u) -- cycle; 
\fill[black, opacity=0.1] (o) -- (d) -- (q) -- cycle; 
\draw[] (n) -- (u); 
\draw[] (d) -- (q); 

\node at (o) [circle, fill, inner sep=1pt, label={[label distance=-1mm]above:{\small \( o \)}}] {};
\node at (q) [circle, fill, inner sep=1pt, label={[label distance=-1mm]above right:{\small \(  \)}}] {};
\node at (n) [circle, fill, inner sep=1pt, label={[label distance=-1mm]left:{\small \(  \)}}] {};
\node at (u) [circle, fill, inner sep=1pt, label={[label distance=-2.5mm]below left:{\small \(  \)}}] {};
\node at (d) [circle, fill, inner sep=1pt, label={[label distance=-1.5mm]above right:{\small \(  \)}}] {};

\node at (wa) [circle, fill, inner sep=1pt, label={[label distance=-1.5mm]above right:{\small \(  \)}}] {};
\node at (wb) [circle, fill, inner sep=1pt, label={[label distance=-1.5mm]above right:{\small \(  \)}}] {};
\node at (w) [circle, fill, inner sep=1pt, label={[label distance=-1.5mm]below:{\small \( w \)}}] {};

\node at (pi) [label={[label distance=-1.5mm, color=gray]below:{\small \( \boldsymbol{\pi} \)}}] {};
\node at (sigma) [label={[label distance=-1.5mm, color=gray]below:{\small \( \boldsymbol{\sigma} \)}}] {};

\node at (A) [label={[label distance=-1.5mm]left:{\small \( A \)}}] {};
\node at (B) [label={[label distance=-1mm]above right:{\small \( B \)}}] {};
\end{tikzpicture}

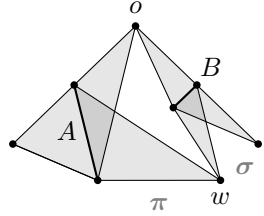
\captionof{figure}{Visualization for the second half of the proof of Lemma \ref{PolarSpaceSecond}.}
\label{rankthreesecondkindforw}
\end{center}

Now let $w\perp o$. If $w \in \alpha \cup \beta$, we define $\phi(w)=w$. If $w \not\in \alpha \cup \beta$, there are two possibilities. First, suppose $A:=\proj_\alpha(w)$ and $B:=\proj_\beta(w)$ are not coplanar. For an arbitrary point $a\in A\setminus\{o\}$, we define $\phi(w)=\eta_a(aw)\cap \<w,B\>$. Similarly, as in the previous paragraph, we find that this is independent of $a\in A\setminus\{o\}$, and also $\phi(w)=\eta_b(bw)\cap\<w,A\>$ for all $b\in B\setminus\{o\}$.

Secondly, suppose $A$ and $B$, as defined above, are coplanar. Then for each $a\in A\setminus\{o\}$, we have $\eta_a(aw)=aw$, and we define $\phi(w)=w$.

Hence, in all cases, we have that if $c\in \alpha\cup\beta\setminus\{o\}$ is collinear to $w$, then $\eta_c(cw)=c\phi(w)$.

It remains to show that lines are mapped to lines. If a line $K$ is not collinear to $o$, then it is collinear to unique distinct points $a\in \alpha$ and $b\in\beta$, and the images under $\phi$ of the points of $K$ are all contained in the line $K' := \eta_a(\<a,K\>)\cap\eta_b(\<b,K\>)$. 
Applying the inverse mapping (obtained by interchanging the roles of $p$ and $p'$), we see that $\phi(K)$ coincides with the whole line $K'$.

Finally, let $K$ be collinear to $o$. Each point $w\in K$ is mapped onto a point $\phi(w)\in ow$. Consider a plane $\pi\not\ni o$ through $K$. By the foregoing, the image $\phi(\pi)$ is contained in a plane that intersects $\<o,K\>$ in the set $\{\phi(w)\mid w\in K\}$. Since distinct planes sharing at least two points intersect in lines, we see that $\phi(K)$ is a line.

This completes the proof of the lemma.
\end{proof}

\begin{coro}\label{PolarMoufang}
Let $\Delta$ be a polar space of rank $3$. Then $\Delta$ is Moufang.
Furthermore, every root elation in each point residue can be written as an even self-projectivity of length $4$.
\end{coro}

\begin{proof}
The claim follows from Corollary \ref{PolarSpaceFirst} and Lemma \ref{PolarSpaceSecond}.
\end{proof}

\subsection{Extending to rank $n > 3$} 

\begin{coro}\label{PolarRanknMoufang}
Let $\Delta$ be a polar space of rank $n$. Then $\Delta$ is Moufang.
\end{coro}

For the induction, suppose we have shown that the elations $\theta$ extend in the case where the polar space $\Delta$ has rank $n-1$. Suppose $\tilde{\Delta}$ is a polar space of rank $n$ containing $\Delta$. Let $\mathcal{A}$ be some apartment of $\tilde{\Delta}$ spanned by a polar frame
$$\{p_{-n},p_{-n+1},\ldots,p_{-1},p_1,p_2,\ldots,p_n\}.$$ 
Since we know that the extension works for $\Delta$, we know that $\theta$ is determined for all singular subspaces spanned by the points
$$\{p_{-n+1},p_{-n+2},\ldots,p_{-1},p_1,p_2,\ldots,p_{n-1}\}.$$ 
Now, with the exact same techniques as before, we can copy the action of $\theta$ to the surrounding singular subspaces.

\section{Polar spaces of infinite rank} 

We will now use Definition \ref{MoufangDefnInfty} to prove the Moufang property for polar spaces of infinite rank.
This section is the only section in which degenerate polar spaces appear. 
Hence, we explicitly talk of non-degenerate polar spaces, instead of just polar spaces.
We refer the reader interested in polar space of infinite rank to \cite{Pasini:09}.
For our proof, we will only need the following two facts.

\begin{fact}\label{everinfiniteembedd}
Every non-degenerate polar space of infinite rank is embeddable into some projective space (cf.~\cite[Theorem~5.1]{Joh:90}).
\end{fact}

\begin{fact}\label{projectivesubspacepolarspaceintersection}
Let $\widehat{\Delta}$ be a projective space and let $\Delta$ be a non-degenerate polar space embedded in $\widehat{\Delta}$.
Then every subspace of $\widehat{\Delta}$ intersects $\Delta$ in a (possibly degenerate) polar subspace. 
This follows straight from the definition.
\end{fact}

\begin{lem}\label{existfiniteranksubpolarspace}
Let $\Delta$ be a non-degenerate polar space of infinite rank. Then for every finite set of points $P$ in $\Delta$ there exists a finite-rank polar subspace $\Gamma \subset \Delta$ containing all points of $P$.
\end{lem}

\begin{proof}
According to Fact \ref{everinfiniteembedd}, we can embed $\Delta$ into some projective space $\widehat{\Delta}$. 
In $\widehat{\Delta}$ we consider the subspace $\<P\>$ generated by all points of $P$. 
Since $P$ is finite, $\<P\>$ has finite dimension in $\widehat{\Delta}$.
The intersection of $\<P\>$ with $\Delta$ is a polar subspace of $\Delta$ of finite rank, but possibly degenerate (Fact \ref{projectivesubspacepolarspaceintersection}).
Suppose $\<P\>_\Delta := \<P\> \cap \Delta$ is degenerate, and let $R$ be the non-empty radical of $\<P\>_\Delta$.
Let $S$ be a singular subspace opposite $R$. The subspace $S$ is finite-dimensional, because $R$ is finite-dimensional.
Now let $\Gamma$ be the polar subspace of $\Delta$ generated by $\<P\>_\Delta$ and $S$. 
Then $\Gamma$ is of finite rank and non-degenerate.
\end{proof}

\begin{lem}\label{nondegeneratesubpolarwithunion}
Let $\Delta$ be a non-degenerate polar space of infinite rank, and let $\Gamma$ and $\Gamma'$ be two distinct polar subspaces of finite rank.
Then there exists a polar subspace $\Delta'$ of finite rank containing $\Gamma \cup \Gamma'$.
\end{lem}

\begin{proof}
We consider the subspace $S$ of $\Delta$ generated by $\Gamma \cup \Gamma'$. 
Then $S$ is a possibly degenerate polar space of finite rank.
Suppose $S$ is degenerate, and let $R$ be the radical of $S$.
Let $T$ be a singular subspace opposite $R$, and let $\Delta'$ be the polar subspace generated by $S$ and $T$.
Then $\Delta'$ is a polar subspace of finite rank containing $\Gamma \cup \Gamma'$.
\end{proof}

\begin{theorem}\label{InfinitePolarSpacesMoufang}
Let $\Delta$ be a non-degenerate polar space of infinite rank. Then $\Delta$ is Moufang.
\end{theorem}

\begin{proof}
\textbf{(Line elations)}
Let $p$ and $q$ be two collinear points in $\Delta$. Let $a$ be a point in $\Delta$ collinear to $p$ but not to $q$.
Let $a'$ be some point on the line $pa$ not equal to $p$ or $a$.
We want to show that we can find a line elation of $\Delta$ with center $(p,pq,q)$ which maps $a$ to $a'$.

Let $x$ be some arbitrary point in $\Delta$.
For the set of points $P := \{a,a',p,q,x\}$, we can find a non-degenerate polar subspace $\Gamma \subset \Delta$ of finite rank containing all points of $P$ (Lemma \ref{existfiniteranksubpolarspace}).
For $\Gamma$ there exists a line elation $\theta_\Gamma$ with center $(p,pq,q)$ which maps $a$ to $a'$ (Corollary \ref{PolarRanknMoufang}).

i) We first want to show that the image of $x$ under $\theta_\Gamma$ does not actually depend on the polar subspace $\Gamma \subset \Delta$.
For that, let $\Gamma'$ be another non-degenerate polar subspace in $\Delta$ of finite rank containing all points of $P$.
According to Lemma \ref{nondegeneratesubpolarwithunion}, there exists a non-degenerate polar subspace $\Lambda$ of finite rank containing $\Gamma \cup \Gamma'$.
Now $\theta_\Lambda|_{\Gamma}$ is $\theta_\Gamma$, and $\theta_{\Lambda}|_{\Gamma'}$ is $\theta_{\Gamma'}$, and both are restrictions of $\theta_{\Lambda}$,
and therefore, for every point $x \in \Gamma \cap \Gamma'$, we have $\theta_\Gamma(x)=\theta_{\Gamma'}(x)$. 
Thus, the image of $x$ under $\theta_\Gamma$ does not depend on the polar subspace $\Gamma$, and we just write $\theta$ in the following.

ii) The only thing left to show is that $\theta$ maps lines to lines.
For that, let $L$ be an arbitrary line in $\Delta$.
Then $L$ is contained in a non-degenerate polar subspace $\Gamma$ of $\Delta$ of finite rank.
We know that $\theta|_\Gamma(L)$ is a line, and that this image does not depend on $\Gamma \subset \Delta$.

Thus, $\theta$ maps lines to lines, is defined for every point $x \in \Delta$, has center $(p,pq,q)$, and maps $a$ to $a'$.

\textbf{(Point elations)}

For point elations we can use the exact same arguments as for line elations.
\end{proof}


\subsection*{Acknowledgment} The author is grateful to Prof. Linus Kramer for the idea for this project. Furthermore, the author wants to thank Prof. Linus Kramer and Prof. Hendrik Van Maldeghem for their general advice and support. In particular, the discussions with Prof. Hendrik Van Maldeghem about the subjects of this article have been very insightful. The author would also like to thank the anonymous referee for suggesting the definitions of point elations and line elations, and for suggesting trying to prove the infinite-rank case.

The author is funded by the Claussen-Simon-Stiftung and by the Deutsche Forschungsgemeinschaft (DFG, German Research Foundation) under Germany's Excellence Strategy EXC 2044 --390685587, Mathematics Münster: Dynamics--Geometry--Structure. This work is part of the PhD project of the author.

\end{document}